\documentclass[12pt]{article}
\usepackage{amsmath, amssymb, amsthm}
\usepackage{graphicx}
\usepackage{hyperref}
\usepackage{cite}
\usepackage{mathrsfs}
\usepackage{enumitem}
\usepackage{geometry}
\usepackage{empheq}
\usepackage{mathtools}
\usepackage{varwidth}
\newtheorem{theorem}{Theorem}[section]
\newtheorem{lemma}[theorem]{Lemma}

\newtheorem{definition}[theorem]{Definition}

\newtheorem{remark}[theorem]{Remark}

\newtheorem{proposition}[theorem]{Proposition}

\begin{document}
\title{Weighted and unweighted regularity of bilinear pseudo-differential operators with symbols in general H\"{o}rmander classes.
}

\author{
      Guangqing Wang 
}




\date{}
\maketitle

\begin{abstract}
This paper investigates the boundedness of bilinear pseudo-differential operators
with symbols in the H\"{o}rmander class $BS_{\varrho,\delta}^m(\mathbb{R}^n)$ in the
previously unexplored regime $0 \leq \varrho < \delta < 1$. We establish boundedness from $H^p(\mathbb{R}^n) \times H^q(\mathbb{R}^n)$ to
$L^r(\mathbb{R}^n)$ (with $L^r$ replaced by $\mathrm{BMO}$ when $p=q=r=\infty$) under
the probably optimal condition on the order
$$m \leq m_\varrho(p,q) - \frac{n\max\{\delta-\varrho,0\}}{\max\{r,2\}},$$
where $m_\varrho(p,q)$ is the critical order in the case $0\leq\delta\leq\varrho<1.$

Furthermore, we develop refined pointwise estimates via sharp maximal functions,
establishing that for $m \leq -n(1-\varrho)(\frac{1}{\min\{r_1,2\}}+ \frac{1}{\min\{r_2,2\}})$ with $1<r_{1},r_{2}<\infty$,
the bilinear operators satisfy
$$M^\sharp T_a(f_1,f_2)(x) \lesssim \mathcal{M}_{\vec{r}}(f_1,f_2)(x).$$
This extends the parameter range from the restrictive condition
$0 \leq \delta \leq \varrho < 1$ to the general setting
$0 \leq \varrho \leq 1$, $0 \leq \delta < 1$ with $\delta > \varrho$ permitted,
and generalizes previous results of Park and Tomita to distinct exponent pairs.
Consequently, we obtain weighted norm inequalities for bilinear pseudo-differential
operators under multilinear $A_{\vec{p},(\vec{r},\infty)}$ weights.



\end{abstract}

\section{Introduction and Main results}

The theory of pseudo-differential operators, initiated by H\"{o}rmander \cite{Hormander1967}, has been a central subject in harmonic analysis and partial differential equations. For the linear case, the boundedness properties of operators defined by
\begin{eqnarray}\label{D1}
T_{a}f(x)
=\int_{\mathbb{R}^{n} }e^{ i x\cdot\xi}a(x,\xi) \hat{f}(\xi)d\xi
\end{eqnarray}
with symbols in the H\"{o}rmander class $S^{m}_{\varrho,\delta}$ are well understood, where $m \in \mathbb{R}$ and $0 \leq \varrho, \delta\leq 1$. The classical Cotlar-Stein lemma establishes that operators with symbols in $S^{-\frac{n}{2}\max\{\delta-\varrho,0\}}_{\varrho,\delta}$ for $0 \leq \varrho \leq 1,0\leq\delta<1$ are bounded on $L^{2}(\mathbb{R}^{n})$\cite{CalderonVaillancourt1972,Hounie2}, and this result was subsequently extended to $L^{p}(1<p<\infty)$ spaces provided
$$m\leq m_{\varrho}(p)-n\frac{\max\{\delta-\varrho,0\}}{\max\{p,2\}},$$
where $m_{\varrho}(p)=-n(1-\varrho)|\frac{1}{2}-\frac{1}{p}|$
\cite{Fefferman1973,ChanilloTorchinsky1985,W}. It is noteworthy that the order condition $m \le m_{\varrho}(p)$ is sharp when $0 \le \delta \le \varrho < 1$.
However, the question of whether the bound on $m$ is optimal remains unresolved in the case when $p\neq2$ and $0\leq\varrho<\delta<1$.

The multilinear extension of this theory, particularly the bilinear case, has attracted considerable attention due to its intrinsic mathematical interest and applications in nonlinear PDEs.
The study of bilinear pseudo-differential operators defined by
\begin{eqnarray}\label{Df}
T_{a}(f,g)(x)
&=&\int_{\mathbb{R}^{2n} }e^{ i x\cdot(\xi+\eta)}a(x,\xi,\eta) \hat{f}(\xi) \hat{g}(\eta)d\xi d\eta
\end{eqnarray}
with symbols in the bilinear H\"{o}rmander class $BS^m_{\varrho,\delta}$ was initiated by B\'{e}nyi et al. \cite{BenyiMaldonadoNaiboTorres2010}.
This symbol class $BS^{m}_{\varrho,\delta}$ consists of smooth functions $\sigma(x,\xi,\eta)$ satisfying
\[
|\partial^{\alpha}_{x}\partial^{\beta}_{\xi}\partial^{\gamma}_{\eta}\sigma(x,\xi,\eta)| \leq C_{\alpha,\beta,\gamma}(1+|\xi|+|\eta|)^{m+\delta|\alpha|-\varrho(|\beta|+|\gamma|)}
\]
for all multi-indices $\alpha, \beta, \gamma \in \mathbb{N}^{n}.$ Although the symbolic calculus for these operators was established\cite{BenyiMaldonadoNaiboTorres2010}, the bilinear theory exhibits markedly different behavior from its linear counterpart. A striking phenomenon, observed by B\'{e}nyi and Torres \cite{BenyiTorres2004}, is that the condition $\sigma \in BS^{0}_{0,0}$  does not guarantee any boundedness for the associated bilinear operator. This fundamental gap necessitates a more refined analysis of the critical order $m_{\varrho}(p,q)$ for boundedness from $H^{p} \times H^{q}$ to $L^{r}$ with $1/p + 1/q = 1/r,$
where
\begin{equation}\label{m}
m_\varrho(p,q)=(1-\varrho)m_0(p,q),
\end{equation}
and
\begin{equation*}
m_0(p,q)=-n\max\left\{\frac{1}{2},\frac{1}{p},\frac{1}{q},1-\frac{1}{r},\frac{1}{r}-\frac{1}{2}\right\}.
\end{equation*}

The systematic study of critical order boundedness has been developed through several important works. For the case $\varrho = 0$, Miyachi and Tomita \cite{MiyachiTomita2013} established the complete characterization of boundedness in the full range $0 < p, q, r \leq \infty$. For $0 < \varrho < 1$, the subcritical case $m < m_{\varrho}(p,q)$ was treated by Michalowski et al. \cite{MichalowskiRuleStaubach2014} and B\'{e}nyi et al. \cite{BenyiBernicotMaldonadoNaiboTorres2013}. The critical case $m = m_{\varrho}(p,q)$ for $1 \leq p, q, r \leq \infty$ was resolved by Miyachi and Tomita \cite{MiyachiTomita2020}, while the extension to the full range $0 < p, q \leq \infty$ was completed in \cite{MiyachiTomita2018} through delicate atomic decomposition techniques and new methods covering both $\varrho = 0$ and $0 < \varrho < 1$ simultaneously.
The works of Miyachi and Tomita \cite{MiyachiTomita2020, MiyachiTomita2018} primarily focus on the exotic case $\delta = \varrho$ with $0 \leq \varrho < 1$. The more general case where $0 \leq \varrho \leq 1$ and $0 \leq \delta < 1$ without the restriction $\delta \leq \varrho$ has not been fully investigated. This is particularly significant because when $\delta > \varrho$, the symbol class exhibits different behavior that requires new techniques to handle. The purpose of the present paper is to solve the problem in the range $0\leq\varrho<\delta<1.$

\begin{theorem}\label{thm:main}
Let $0\leq\varrho\leq1, 0\leq\delta<1,$ $0<p,q\leq \infty$, $\frac{1}{p}+\frac{1}{q}=\frac{1}{r}$ and
$$m\leq m_\varrho(p,q)-n\frac{\max\{\delta-\varrho,0\}}{\max\{r,2\}},$$
where $m_\varrho(p,q)$ is given by (\ref{m}).

Then all bilinear
pseudo-differential operators with symbols in $BS_{\varrho,\delta}^m(\mathbb{R}^n)$ are bounded from
$H^p(\mathbb{R}^n) \times H^q(\mathbb{R}^n)$ to $L^{r}(\mathbb{R}^n),$ where $L^{r}(\mathbb{R}^n)$ should be replaced
by $BMO(\mathbb{R}^n)$ if $p=q=r=\infty.$
\end{theorem}

\begin{remark}
The optimality of the bound on $m$ in the regime $0 \leq \varrho < \delta < 1$ remains an open question, paralleling the situation in the linear theory.
\end{remark}

The novel contribution of Theorem~\ref{thm:main} lies in the treatment of the case $0 \leq \varrho < \delta < 1$, which was previously investigated by Rodr\'{\i}guez-L\'{o}pez and Staubach \cite{Staubach} under more restrictive conditions on $m$. Specifically, they established boundedness for $1 \leq p, q < \infty$, $0 < \varrho, \delta \leq 1$ with $\delta < 1$, under the hypothesis
\[
m < n(\rho - 1)\left[\max\left(\left|\frac{1}{2} - \frac{1}{q}\right|, \left|\frac{1}{2} - \frac{1}{q}\right|\right) + \frac{1}{\min(2, q, q)}\right] - \frac{\max\{\delta - \varrho, 0\}}{2}.
\]

Theorem~\ref{thm:main} is established through complex interpolation based on the following foundational estimates:
\begin{center}
\begin{varwidth}{\linewidth}
\begin{itemize}
  \item [(i)] $L^2(\mathbb{R}^n) \times L^2(\mathbb{R}^n)\rightarrow L^{1}(\mathbb{R}^n).$
  \item [(ii)] $L^2(\mathbb{R}^n) \times L^\infty(\mathbb{R}^n)\rightarrow L^{2}(\mathbb{R}^n).$
  \item [(iii)] $H^p(\mathbb{R}^n) \times L^2(\mathbb{R}^n) \rightarrow L^{r}(\mathbb{R}^n)$ for $0<p<1$.
  \item [(vi)] $H^p(\mathbb{R}^n) \times L^\infty(\mathbb{R}^n) \rightarrow L^{p}(\mathbb{R}^n)$ for $0<p<1$.
  \item [(v)] $L^\infty(\mathbb{R}^n) \times L^\infty(\mathbb{R}^n) \rightarrow BMO(\mathbb{R}^n)$.
\end{itemize}
\end{varwidth}
\end{center}
The proofs for the case $0 \leq \delta \leq \varrho < 1$ can be found in \cite{MiyachiTomita2013, BenyiBernicotMaldonadoNaiboTorres2013, MiyachiTomita2018}, and are therefore omitted here.

\begin{theorem}\label{TM}
Let $0\leq\varrho\leq1, 0\leq\delta<1$ and $m \leq -\frac{n}{2}(1-\varrho)-\frac{n}{2}\max\{\delta-\varrho,0\}$. Then all bilinear
pseudo-differential operators with symbols in $BS_{\varrho,\delta}^m(\mathbb{R}^n)$ are bounded from
$L^2(\mathbb{R}^n) \times L^2(\mathbb{R}^n)$ to $L^{1}(\mathbb{R}^n)$.
\end{theorem}

\begin{theorem}\label{TM2}
Let $0\leq\varrho\leq1, 0\leq\delta<1$ and $m \leq -\frac{n}{2}(1-\varrho)-\frac{n}{2}\max\{\delta-\varrho,0\}$. Then all bilinear
pseudo-differential operators with symbols in $BS_{\varrho,\delta}^m(\mathbb{R}^n)$ are bounded from
$L^2(\mathbb{R}^n) \times L^\infty(\mathbb{R}^n)$ to $L^{2}(\mathbb{R}^n)$.
\end{theorem}
\begin{remark}
The symbolic calculus developed in \cite{BenyiMaldonadoNaiboTorres2010} implies that Theorem~\ref{TM} follows from Theorem~\ref{TM2} when $0 \leq \delta \leq \varrho < 1$. However, no such calculus is available for the regime $1 > \delta > \varrho \geq 0.$
\end{remark}

\begin{theorem}\label{T2}
Let $0\leq\varrho\leq1, 0\leq\delta<1,$ $0<p<1$, $\frac{1}{p}+\frac{1}{2}=\frac{1}{r}$ and $m \leq -\frac{n(1-\varrho)}{p}-\frac{n}{2}\max\{\delta-\varrho,0\}$. Then all bilinear
pseudo-differential operators with symbols in $BS_{\varrho,\delta}^m(\mathbb{R}^n)$ are bounded from
$H^p(\mathbb{R}^n) \times L^2(\mathbb{R}^n)$ to $L^{r}(\mathbb{R}^n)$.
\end{theorem}

\begin{theorem}\label{T4}
Let $0\leq\varrho\leq1, 0\leq\delta<1,$ $0<p<1$ and $m \leq -\frac{n(1-\varrho)}{p}-\frac{n}{2}\max\{\delta-\varrho,0\}$. Then all bilinear
pseudo-differential operators with symbols in $BS_{\varrho,\delta}^m(\mathbb{R}^n)$ are bounded from
$H^p(\mathbb{R}^n) \times L^\infty(\mathbb{R}^n)$ to $L^{p}(\mathbb{R}^n)$.
\end{theorem}

\begin{theorem}\label{T1}
Let $0\leq\varrho< 1, 0\leq\delta<1$ and $m \leq -n(1-\varrho)$. Then all bilinear
pseudo-differential operators with symbols in $BS_{\varrho,\delta}^m(\mathbb{R}^n)$ are bounded from
$L^\infty(\mathbb{R}^n) \times L^\infty(\mathbb{R}^n)$ to $BMO(\mathbb{R}^n)$.
\end{theorem}

\begin{remark}
Theorem~\ref{T1} follows from the pointwise estimate \eqref{Point} (whose proof is omitted here). This result was previously established by Naibo \cite{Naibo} for $0 < \delta \leq \varrho < \frac{1}{2}$, and subsequently extended by Miyachi and Tomita \cite{MiyachiTomita2020} to the broader regime $0 \leq \delta \leq \varrho < 1$. For multilinear generalizations, we refer to the recent works of Park and Tomita \cite{ParkTomita2024, ParkTomita2025}.
\end{remark}

\begin{remark}
For $0 \leq \varrho < \delta < 1$, the condition $m \leq -n(1-\varrho)$ is sharp for boundedness from $L^\infty(\mathbb{R}^n) \times L^\infty(\mathbb{R}^n)$ to $BMO(\mathbb{R}^n)$; see \cite{MiyachiTomita2020}. However, the optimality of the bound $m \leq -n(1-\varrho) - \frac{n(\delta-\varrho)}{2}$ for boundedness from $H^1(\mathbb{R}^n) \times L^\infty(\mathbb{R}^n)$ to $L^1(\mathbb{R}^n)$ remains an open problem.
\end{remark}

Parallel to these developments, the weighted norm inequality emerges as an inevitable research topic. The multilinear weight theory was systematically developed by Lerner et al. \cite{LOPTT}, and subsequently extended by Li et al. \cite{Li2020} and Nieraeth \cite{Nieraeth}. The multilinear Muckenhoupt class $A_{\vec p}$ was first introduced in \cite{LOPTT} to characterize the boundedness
of multilinear Calder\'{o}n-Zygmund operators. This weight class can be characterized through the multilinear maximal function
\begin{equation}
\mathbf{M}_{r}(f_{1},\ldots,f_{l})(x)
=\sup_{Q\ni x}\biggl(\frac{1}{|Q|^{l}}\int_{Q^{l}}\prod_{i=1}^{l}|f_{i}(y_{i})|^{r}\,d\vec{y}\biggr)^{\frac{1}{r}}.
\end{equation}
Later, Li et al. \cite{Li2020} extended this class to the more flexible class $A_{\vec p,\vec r}$, which plays a fundamental role in multilinear extrapolation theory and sparse domination techniques. More recently, Nieraeth \cite{Nieraeth} introduced an even more general family of weights $A_{\vec p,(\vec r,s)}$, whose special case $A_{\vec p,(\vec r,\infty)}=A_{\vec p,\vec r}$ is characterized by
\begin{equation}\label{vector-maximal}
\mathcal{M}_{\vec{r}}(f_{1},\ldots,f_{l})(x)
=\sup_{Q\ni x}\prod_{i=1}^{l}\biggl(\frac{1}{|Q|}\int_{Q}|f_{i}(y_{i})|^{r_{i}}\,dy_{i}\biggr)^{\frac{1}{r_{i}}},
\end{equation}
to accommodate estimates with different output exponents and to handle a wider class of multilinear operators.
This hierarchy reflects the evolution of multilinear weighted theory from classical settings to more flexible frameworks suitable for modern extrapolation and quantitative weighted estimates.

Pointwise estimates and the characterization of weight classes via relevant maximal functions have proven to be powerful tools for establishing weighted norm inequalities. For a function $f\in L^{1}_{\mathrm{loc}}(\mathbb{R}^{n})$, the Fefferman-Stein sharp maximal function and the $L^{r}$-version of the Hardy-Littlewood maximal function are defined respectively by
$$M^{\sharp}f(x)=\sup\limits_{x\in Q}\inf\limits_{c}\frac{1}{|Q|}\int_{Q}|f(y)-c|dy\quad {\rm and}\quad M_{r}f(x)=\sup\limits_{x\in Q}(\frac{1}{|Q|}\int_{Q}|f(y)|^{r}dy)^{\frac{1}{r}},$$
where the supremum is taken over all cubes $Q$ containing $x$ with sides parallel to the coordinate axes, and $c$ ranges over all complex numbers. In the linear setting, the pointwise estimate
\begin{equation}\label{E25}
M^{\sharp}(T_{a}f)(x) \lesssim M_{r}f(x)
\end{equation}
for pseudo-differential operators $T_{a}$ with $a \in S^{m}_{\varrho,\delta}$  was subsequently obtained in \cite{ChanilloTorchinsky1985,MiyachiYabuta1987,ParkTomita2024,Wang},where $0 \leq \varrho \leq 1,0 \leq \delta < 1, m \leq-n(1-\varrho)/r$ and $1 < r \leq 2$. In the multilinear setting, Park and Tomita \cite{ParkTomita2024,ParkTomita2025} proved the following pointwise estimate via multilinear maximal function $\mathbf{M}_{r}$.
\begin{theorem}\cite{ParkTomita2024,ParkTomita2025}\label{PT}
Let $0\leq\rho<1$, $1<r\leq 2$, and $m=-\dfrac{nl}{r}(1-\rho)$. Then every $a\in \mathbb{M}_l S_{\rho,\rho}^m(\mathbb{R}^n)$ satisfies
\begin{equation}\label{E29}
\mathcal{M}_{r/l}^{\sharp}\bigl(T_{a}(f_1,\ldots,f_l)\bigr)(x) \lesssim \mathbf{M}_r(f_1,\ldots,f_l)(x), \quad x\in\mathbb{R}^n
\end{equation}
for all $f_1,\ldots,f_l\in\mathcal{S}(\mathbb{R}^n)$.
\end{theorem}
As an application of this result, weighted estimates for multilinear pseudo-differential operators with respect to $A_{\vec{p}}$ can be readily established. While the maximal function $\mathbf{M}_{r}$ provides a natural framework for studying multilinear operators with uniform integrability exponents, it becomes inadequate when dealing with distinct exponent tuples $(r_1, \ldots, r_l)$. In such cases, it is more natural to employ maximal operators adapted to the individual integrability properties of each function, rather than relying on a uniform $L^r$-based maximal function. Thus, it is natural to seek pointwise estimates analogous to \eqref{E29} for the operators $\mathcal{M}_{\vec{r}}$. In this work, we establish such results for the bilinear case, as the requisite multilinear estimates corresponding to Theorem~\ref{TM} remain unavailable in full generality.

\begin{theorem}\label{T6}
Let $0\leq\varrho\leq1, 1\leq\delta<1,  1< r_{1},r_{2}<\infty $, $\vec{r}=(\min\{r_{1},2\},\min\{r_{2},2\})$. Then all bilinear
pseudo-differential operators with symbols in $BS_{\varrho,\delta}^m(\mathbb{R}^n)$
satisfy
\begin{equation}\label{Point}
|M^{\sharp}T_{a}(f_{1},f_{2})(x)|\lesssim \mathcal{M}_{\vec{r}}(f_{1},f_{2})(x)
\end{equation}
provided
$$m \leq-n(1-\varrho)(\frac{1}{\min\{r_{1},2\}}+\frac{1}{\min\{r_{2},2\}}).$$
\end{theorem}

\begin{remark}
Theorem \ref{T1} follows from Theorem~\ref{T6} with taking $r_1 = r_2=2$. So its proof will be omitted.
\end{remark}

\begin{remark}
In Theorem~\ref{PT}, $\mathbb{M}_l S_{\rho,\delta}^m(\mathbb{R}^n)$ denotes the multilinear H\"{o}rmander symbol class. Thus $\mathbb{M}_2 S_{\rho,\delta}^m(\mathbb{R}^n) = BS_{\varrho,\delta}^m(\mathbb{R}^n)$, and Theorem~\ref{Point} improves upon Theorem~\ref{PT} in two significant aspects for the bilinear case: first, we extend the parameter range from the restrictive condition $0 \leq \delta \leq \varrho < 1$ to the general setting $0 \leq \varrho \leq 1$, $0 \leq \delta < 1$; second, Theorem~\ref{PT} corresponds to the special case $r_1 = r_2$ of Theorem~\ref{T6}.
\end{remark}

The following theorem presents the characterization of the class $A_{\vec{p},(\vec{r},s)}$ via the maximal function $\mathcal{M}_{\vec{r}}$; see Section~2 below for the relevant notation.
\begin{theorem}\cite[Proposition 2.14]{Nieraeth}
Let $r_1, \ldots, r_m \in (0, \infty)$, $p_1, \ldots, p_m \in (0, \infty]$ with $(\vec{r}, \infty) \leq \vec{p}$ and let $w_1, \ldots, w_m$ be weights with $w = \prod_{j=1}^m w_j$. Then the following are equivalent:
\begin{itemize}
    \item[\rm (i)] $\vec{w} \in A_{\vec{p},(\vec{r},\infty)};$
    \item[\rm (ii)] $\|M_{\vec{r}}\|_{L^{p_1}(w_1^{p_1}) \times \cdots \times L^{p_m}(w_m^{p_m}) \to L^{p,\infty}(w^p)} < \infty$.
\end{itemize}
Moreover, if $\vec{r} < \vec{p}$, then {\rm (i)} and {\rm (ii)} are equivalent to
\begin{itemize}
    \item[\rm (iii)] $\|M_{\vec{r}}\|_{L^{p_1}(w_1^{p_1}) \times \cdots \times L^{p_m}(w_m^{p_m}) \to L^p(w^p)} < \infty.$
\end{itemize}
\end{theorem}

Here and subsequently, we employ the notation $\|f\|_{L^{p}(\omega^{p})}:=\|f\omega\|_{L^{p}}$ for $0<p\leq\infty$. Consequently, one obtains weighted estimates for bilinear pseudo-differential operators with general symbols in $BS_{\varrho,\delta}^{m}(\mathbb{R}^n)$.
\begin{theorem}
Let $0\leq\varrho\leq1, 0\leq\delta<1, r_{1},r_{2}\in (1, \infty)$,
$$m \leq-n(1-\varrho)(\frac{1}{\min\{r_{1},2\}}+\frac{1}{\min\{r_{2},2\}}),$$
$p_1, p_2 \in (0, \infty]$, $\vec{r}=(\min\{r_{1},2\},\min\{r_{2},2\}), \vec{p}=(p_{1},p_{2})$
with $(\vec{r}, \infty) \leq \vec{p}$ and let $w_1,w_2$ be weights with $w =w_{1}w_{2},$ $\vec{w} \in A_{\vec{p},(\vec{r},\infty)}.$ Then all bilinear
pseudo-differential operators with symbols in $BS_{\varrho,\delta}^{m}(\mathbb{R}^n)$
satisfy
$$\|T_{a}\|_{L^{p_1}(w_1^{p_1}) \times  L^{p_2}(w_2^{p_2}) \to L^{p,\infty}(w^p)} < \infty$$
Moreover, if $\vec{r} < \vec{p}$, then
$$\|T_{a}\|_{L^{p_1}(w_1^{p_1}) \times L^{p_2}(w_2^{p_2}) \to L^p(w^p)} < \infty.$$

\end{theorem}

The organization of this paper is as follows. In Section 2, we recall some decomposition on symbol, and necessary preliminaries on $BMO$ spaces, Hardy spaces, and weight functions. Section 3 and Section 4 establish the $L^{1}$-based (Theorem \ref{TM}) estimates and $L^{2}$-based estimates (\ref{TM2}), which serve as the foundation for subsequent developments. Section 5 treats the Hardy space estimates for $0 < p < 1$ (Theorems \ref{T2} and \ref{T4}). Finally, Section 6 develops the refined pointwise estimate (Theorem \ref{T6}).

\section{Some preliminaries and Decomposition of the symbol}
Throughout this paper, we use the notation $A \lesssim B$ to mean $A \leq CB$ for some constant $C > 0$ independent of the relevant parameters, and $A \approx B$ to mean $A \lesssim B$ and $B \lesssim A$.

In order to state the definition of weigh class $A_{\vec{p},(\vec{r},s)}$, we introduce some notations first. Let $r_{1},...,r_{l}\in(0,\infty)$, $s\in(0,\infty]$ and $p_{1},...,p_{l}\in(0,\infty]$, writing $\vec{r}=(r_{1},...,r_{l})$ and similarly for $\vec{p}$, we write $\vec{r}\leq\vec{p}$ if $r_{j}\leq p_{j}\leq\infty$ for $j\in\{1,...,l\}$. Moreover we write $(\vec{r},s)\leq\vec{p}$ if $\vec{r}\leq\vec{p}$ and $p\leq s$, where $p$ is defined by $\frac{1}{p}=\frac{1}{p_{1}}+...+\frac{1}{p_{l}}$. Similarly, we write  $\vec{r}<\vec{p}$ if $r_{j}< p_{j}$ for $j\in\{1,...,l\}$. Moreover we write $(\vec{r},s)<\vec{p}$ if $\vec{r}<\vec{p}$ and $p< s$. For a measurable set $E\subseteq\mathbb{R}^{n}$ with $0<|E|<\infty$, set
$$(f)_{p,E}:=(\frac{1}{|E|}\int_{E}|f(x)|^{p}dx)^{\frac{1}{p}}.$$

The definition of weigh class $A_{\vec{p},(\vec{r},s)}$ is given as follows.
\begin{definition}\cite{Nieraeth}
Let $r_{1},...,r_{l}\in(0,\infty), s\in(0,\infty]$ and $p_{1},...,p_{l}\in(0,\infty]$ with $(\vec{r},s)\leq\vec{p}.$ Let $\omega_{1},...,\omega_{l}$ be weights and write $\vec{\omega}=(\omega_{1},...,\omega_{l}).$ We say that $\vec{\omega}\in A_{\vec{p},(\vec{r},s)}$ if
$$[\vec{\omega}]_{\vec{p},(\vec{r},s)}:=\sup\limits_{Q}\prod\limits_{i=1}^{l}( \omega^{-1}_{i})_{\frac{1}{\frac{1}{r_{i}}-\frac{1}{p_{i}}},Q}
( \prod\limits_{i=1}^{l}\omega_{i})_{\frac{1}{\frac{1}{p}-\frac{1}{s}},Q}<\infty,$$
where the supremum is taken over all cubes $Q\subseteq\mathbb{R}^{n}.$
\end{definition}

Let $\mathcal{S}(\mathbb{R}^{n})$ denote the Schwartz space of rapidly decreasing smooth functions, and let $\mathcal{S}'(\mathbb{R}^{n})$ be its dual, the space of tempered distributions. For $f \in \mathcal{S}(\mathbb{R}^{n})$, we define the Fourier transform and its inverse by
\[
\hat{f}(\xi) = \int_{\mathbb{R}^{n}} e^{-ix\cdot\xi} f(x)\, dx
\quad \text{and} \quad
\check{f}(x) = \frac{1}{(2\pi)^n} \int_{\mathbb{R}^{n}} e^{ix\cdot\xi} f(\xi)\, d\xi.
\]

Given $m \in L^\infty(\mathbb{R}^{n})$, the associated Fourier multiplier operator $m(D)$ acts on $f \in \mathcal{S}(\mathbb{R}^{n})$ via
\[
m(D)f(x) = \frac{1}{(2\pi)^n} \int_{\mathbb{R}^{n}} e^{ix\cdot\xi} m(\xi)\widehat{f}(\xi)\, d\xi.
\]
The Sobolev space $L^{2}_{s}(\mathbb{R}^{n})$ consists of all $f\in L^2(\mathbb{R}^{n})$ satisfying
\[
\|f\|_{L^{2}_{s}} = \left\|(1+4\pi^{2}|D|^{2})^{\frac{s}{2}}f\right\|_{L^2} < \infty.
\]
We now recall the definitions of Hardy spaces and $BMO$ (see \cite[Chapters~3 and 4]{Stein}). Fix $\phi \in \mathcal{S}(\mathbb{R}^{n})$ with $\int_{\mathbb{R}^{n}} \phi(x)\, dx \neq 0$, and set $\phi_t(x) = t^{-n}\phi(x/t)$. For $0 < p \leq \infty$, the Hardy space $H^p(\mathbb{R}^{n})$ comprises all $f \in \mathcal{S}'(\mathbb{R}^{n})$ satisfying
\[
\|f\|_{H^p} = \left\|\sup_{0<t<\infty} |\phi_t * f|\right\|_{L^p} < \infty.
\]
This definition is independent of the choice of $\phi$, and $H^p(\mathbb{R}^{n}) = L^p(\mathbb{R}^{n})$ when $1 < p \leq \infty$.

For $0 < p \leq 1$, a function $a$ is called an $H^p$-atom if there exists a cube $Q = Q_a$ such that
\begin{equation}\label{eq:atom}
\operatorname{supp} a \subset Q, \quad \|a\|_{L^\infty} \leq |Q|^{-1/p}, \quad \int_{\mathbb{R}^{n}} x^\alpha a(x)\, dx = 0 \text{ for } |\alpha| \leq L-1,
\end{equation}
where $L$ is any fixed integer with $L > n/p - n$ (\cite[p.~112]{Stein}). 

Finally, $BMO(\mathbb{R}^{n})$ consists of locally integrable functions $f$ for which
\[
\|f\|_{BMO} = \sup_{Q} \frac{1}{|Q|} \int_Q |f(x) - f_Q|\, dx < \infty,
\]
where $f_Q$ denotes the average of $f$ over the cube $Q$ and the supremum ranges over all cubes in $\mathbb{R}^{n}$. It is well known that $BMO(\mathbb{R}^{n})$ is the dual space of $H^1(\mathbb{R}^{n})$.

Next, we introduce two types  partition of unity. One is the dyadic decomposition: let \( A = \{\zeta \in \mathbb{R}^{d} : \frac{1}{2} \leq |\zeta| \leq 2\} \) be an annulus, \( \Psi_0(\zeta) \in C^\infty_0(B(0,2)) \), and define \( \Psi_j(\zeta) = \Psi(2^{-j}\zeta) \) for \( j \geq 1 \), where \( \Psi(\zeta) \in C^\infty_0(A) \). Then we have
\begin{equation}\label{E0}
    \Psi_0(\zeta) + \sum_{j=1}^{\infty} \Psi_j(\zeta) = 1 \quad \text{for all} \quad \zeta \in \mathbb{R}^{d}.
\end{equation}
The other is the uniform decomposition: take $\phi\in C^{\infty}_{c}(\mathbb{R}^{n})$ such that $\phi(\xi)=0$ if $|\xi|\geq n$ and
\begin{equation}\label{E7}
\sum\limits_{k}\phi(\xi-k)=1 \quad \text{for all} \quad \xi\in\mathbb{R}^{n}.
\end{equation}
Using these partitions of unity (\ref{E0}) with $d=2n$, the bilinear pseudo-differential operator $T_{a}$ can be decomposed as:
\begin{equation}\label{E1}
   \sum_{j=0}^{\infty}T_{a_{j}}~\mathrm{with}~a_{j}(x,\xi,\eta)=a(x,\xi,\eta)\Psi_j(\xi,\eta)
\end{equation}
and $T_{a_{j}}$ can be decomposed as
\begin{subequations}\label{eq:decomp}
\begin{empheq}[left=\empheqlbrace]{alignat=2}
&\sum_{\nu_{1}\in \mathbb{Z}^{n}} T_{a_{j}^{\nu_{1}}}
&& \text{with } a_{j}^{\nu_{1}}(x,\xi,\eta)=a(x,\xi,\eta)\Psi_j(\xi,\eta) \phi(2^{-j\varrho}\xi-\nu_{1}),
\label{eq:decomp-a} \\[-2mm]
&\sum_{\nu_{2}\in \mathbb{Z}^{n}} T_{a_{j}^{\nu_{2}}}
&& \text{with } a_{j}^{\nu_{2}}(x,\xi,\eta)=a(x,\xi,\eta)\Psi_j(\xi,\eta)\phi(2^{-j\varrho}\eta-\nu_{2}),
\label{eq:decomp-b} \\[-2mm]
&\sum_{\nu_{1},\nu_{2}\in \mathbb{Z}^{n}}T_{a_{j}^{\nu_{1},\nu_{2}}}
&& \text{with } a_{j}^{\nu_{1},\nu_{2}}(x,\xi,\eta)=a(x,\xi,\eta)\Psi_j(\xi,\eta) \phi(2^{-j\varrho}\xi-\nu_{1})\phi(2^{-j\varrho}\eta-\nu_{2}).
\label{eq:decomp-c}
\end{empheq}
\end{subequations}

Denote
$$\Lambda_{j,1}=\{\nu_{1}\in\mathbb{Z}^{n}:{\rm supp}_{\xi}a_{j}^{\nu_{1},\nu_{2}}\cap {\rm supp}_{\xi}\phi(2^{-j\varrho}\cdot-\nu_{1})\},$$
$$\Lambda_{j,2}=\{\nu_{2}\in\mathbb{Z}^{n}:{\rm supp}_{\eta}a_{j}^{\nu_{1},\nu_{2}}\cap{\rm supp}_{\eta}\phi(2^{-j\varrho}\cdot-\nu_{2})\},$$
then
\begin{center}
$|\Lambda_{j,1}|\lesssim 2^{j(1-\varrho)n},|\Lambda_{j,2}|\lesssim 2^{j(1-\varrho)n}.$
\end{center}

The following estimate for square function is useful which can be gotten by a periodization technique and Hausdorff-Young's inequality. See \cite{MiyachiTomita2020} for the case when $p=q=2.$
\begin{lemma}\label{La0}
Let $1\leq p\leq 2,\frac{1}{p'}+\frac{1}{p}=1$ and \(\psi \in \mathcal{S}(\mathbb{R}^n)\). Then
\begin{equation}\label{E4}
\left( \sum_{\ell \in \mathbb{Z}^n} |\psi(\lambda D-\ell)f(x)|^{p'} \right)^{\frac{1}{p'}} \lesssim \big(\int_{\mathbb{R}^{n}}\frac{\lambda^{n}|f(y)|^{p}}{(1+\lambda|x-y|)^{N}}dy\big)^{\frac{1}{p}}
\end{equation}
holds for all \(x \in \mathbb{R}^n\) and positive integer $N>n$.
\end{lemma}

\begin{remark}\label{r1}
Set $\varphi_{N,\lambda}(x)=\frac{\lambda^{n}}{(1+\lambda|x|)^{N}}$ for $\lambda>0$ and $\vec{r}=(r_{1},r_{2})$. It is easy to check
\begin{eqnarray*}
\big(\varphi_{N_{1},\lambda}\ast|f|^{r_{1}}(x)\big)^{\frac{1}{r_{1}}}\big(\varphi_{N_{2},\lambda}\ast|g|^{r_{2}}(x)\big)^{\frac{1}{r_{2}}}
\lesssim\mathcal{M}_{\vec{r}}(f,g)(x)
\end{eqnarray*}
for positive integer $N_{1},N_{2}>n$ and $f,g\in L^{1}_{loc}(\mathbb{R}^{n})$
\end{remark}

For the almost orthogonal functions and operators, we use the following lemma except for Schur's lemma.

\begin{lemma}\cite[Lemma 2.2]{Hounie2}\label{La8}
Let $s>\frac{n}{2}$ be a real number and write $t=\frac{n}{2s}.$ There is a constant $C=C_{n,s}$ such that for any finite sequence $f_{k}\in L^{2}_{s},k\in \mathbb{Z}^{n}$ the following holds
\begin{equation}
\|\sum\limits_{k}e^{i\langle k, \cdot\rangle}f_{k}\|^{2}_{L^{2}}\leq C\big(\sum\limits_{k}\|f_{k}\|^{2}_{L^{2}}\big)^{1-t}\big(\sum\limits_{k}\|f_{k}\|^{2}_{L^{2}_{s}}\big)^{t}.
\end{equation}
\end{lemma}

\section{$L^{1}$ estimates: Proof of Theorem \ref{TM}}

By the symmetry, we assume that
$$supp_{\xi,\eta}a(x,\xi,\eta)\subset\{(\xi,\eta):|\xi|<|\eta|\}.$$
By the fact that
\[
\|T_a(f,g)\|_{L^{1}}=\sup\limits_{\|h\|_{L^{\infty}}\leq1}|\langle T_a(f,g), h \rangle|,
\]
one shall prove the following estimate for the trilinear form:
\[
|\langle T_a(f,g), h \rangle| \lesssim \|f\|_{L^2} \|g\|_{L^2}\|h\|_{L^\infty}.
\]
Without loss of generality, we assume that $h$ has compact support. Decompose $T_a$ as (\ref{E1}), and
define
\begin{eqnarray}\label{E41}
T^{*}_{a}(f,h)(z)
&=&\int_{\mathbb{R}^{2n} }\int_{\mathbb{R}^{2n} } e^{ -i  (x-y)\cdot\xi} e^{ -i  (x-z)\cdot\eta} a(x,\xi,\eta) d\xi d\eta f(y)h(x)dydx.
\end{eqnarray}
Then
\[
\langle T_{a_{j}}(f,g), h \rangle=\langle g , T^{*}_{a_{j}}(f,h)\rangle
\]
and
\begin{eqnarray}
\operatorname{supp}_{\zeta}\widehat{T^{*}_{a_{j}}(f,h)}\subset\{\zeta:|\zeta|\approx 2^{j}\},
\end{eqnarray}
where $a_{j}$ is given by (\ref{E1}).
Decompose $f$ as
\[
f = \sum_{\ell} f_\ell, \quad f_\ell = \psi_\ell(D)f.
\]
We claim that
\begin{equation}\label{a1}
\|T^{*}_{a_{j}}(f_\ell,h)\|_{L^{2}}\lesssim (\max\{1, 2^{\ell-j\varrho}\})^{\frac{n}{2}}2^{-\frac{n}{2}(1-\varrho)}\|f_\ell\|_{L^2}\|h\|_{L^{\infty}}.
\end{equation}
Then the proof can been finished by Schur's lemma.

To get inequality (\ref{a1}), we decompose the operators $T^{*}_{a_{j}}$ with respect to variable $\xi$ as in (\ref{eq:decomp-a}):
\begin{equation*}
    T^{*}_{a_{j}} = \sum\limits_{\nu_{1}\in \mathbb{Z}^{n}} T^{*}_{a_{j}^{\nu_{1}}}.
\end{equation*}
Define
\[
\Lambda_{j,\ell} = \{\nu_1 \in \mathbb{Z}^n : \operatorname{supp} \varphi(2^{-j\varrho} - \nu_1) \cap \operatorname{supp} \psi_\ell \neq \emptyset\}.
\]
Then H\"{o}lder's inequality gives
\begin{eqnarray*}
|T^{*}_{a_{j}}(f_{\ell},h)(z)|
&\leq&|\Lambda_{j,\ell}|^{\frac{1}{2}}(\sum\limits_{\nu_{1}\in\Lambda_{j,\ell}}|T^{*}_{a_{j}^{\nu_{1}}}(f_{\ell},h)(z)|^{2})^{\frac{1}{2}}.
\end{eqnarray*}
By the fact that the number of elements of \(\Lambda_{j,\ell}\) satisfies
\[
|\Lambda_{j,\ell}| \lesssim (\max\{1, 2^{\ell-j\varrho}\})^n.
\]
it follows that
\begin{equation}\label{E5}
\|T^{*}_{a_{j}}(f_{\ell},h)\|^{2}_{L^{2}}
\leq|\Lambda_{j,\ell}|\sum\limits_{\nu_{1}\in \Lambda_{j,\ell}}\|T^{*}_{a_{j}^{\nu_{1}}}(f_{\ell},h)\|^{2}_{L^{2}}
\end{equation}
So, to get the estimate (\ref{a1}), it is enough to prove
\begin{equation}\label{a11}
\sum\limits_{\nu_{1}\in \Lambda_{j,\ell}}\|T^{*}_{a_{j}^{\nu_{1}}}(f_{\ell},h)\|^{2}_{L^{2}}
\lesssim 2^{-n(1-\varrho)}\|f_\ell\|^{2}_{L^2}\|h\|^{2}_{L^{\infty}}.
\end{equation}
To do this, we use the main idea of Hounie's \cite{Hounie2}.

\begin{lemma}\label{La9}
Let $a(x,\xi,\eta)$ be a bounded continuous function with bounded continuous derivatives $\partial_{x}^{\gamma}\partial_{\xi}^{\alpha}\partial_{\eta}^{\beta}a(x,\xi,\eta)$ for $|\alpha|,|\beta|\leq n+1 $, $|\gamma|\leq N_{0}=:[\frac{n}{2}]+1$, and assume
\[
 \operatorname{supp}_{\xi,\eta} a(x,\xi,\eta) \subset\{(\xi, \eta):|\xi|\leq n , |\eta|\leq n\}.
\]
Define
\begin{center}
$p(a)=\sum\limits_{|\alpha|,|\beta|\leq n+1}\|\partial_{\xi}^{\alpha}\partial_{\eta}^{\beta}a\|_{L^{\infty}(\mathbb{R}^{3n})},$
$q(a)=\sum\limits_{|\alpha|,|\beta|\leq n+1,|\gamma|\leq N_{0}}\|\partial_{x}^{\gamma}\partial_{\xi}^{\alpha}\partial_{\eta}^{\beta}a\|_{L^{\infty}(\mathbb{R}^{3n})}$
\end{center}
If $h(z,x):=h_{z}(x)$ is smooth and has compact support with respect to variable $z$ and $x$ respectively. Then the operator $T^{*}_a$ defined by (\ref{E41}) satisfying the following estimates:
\begin{align}
\|T^{*}_a(f,&h_{z})\|^{2}_{L^{2}}\nonumber\\
&\lesssim p^{2}(a)
\int_{\mathbb{R}^{n}}
\int_{\mathbb{R}^{n}}\int_{\mathbb{R}^{n}}
\frac{|f(y)|^{2}}{(1+|y-x|)^{n+1}(1+|z-x|)^{2n+1}}dy dx\|h(z,\cdot)\|^{2}_{L^{2}}
dz; \label{E42}\\
\|T^{*}_a(f,&\partial^{\gamma}h_{z})\|^{2}_{L^{2}}\nonumber\\
&\lesssim q^{2}(a)
\int_{\mathbb{R}^{n}}
\int_{\mathbb{R}^{n}}\int_{\mathbb{R}^{n}}
\frac{|f(y)|^{2}}{(1+|y-x|)^{n+1}(1+|z-x|)^{2n+2}}dy dx\|h(z,\cdot)\|^{2}_{L^{2}}
dz; \label{E43}\\
\|T^{*}_a(f,&h_{z})\|^{2}_{L^{2}_{N_{0}}}\nonumber\\
&\lesssim p^{2}(a)\sum\limits_{|\alpha|\leq N_{0}}\int_{\mathbb{R}^{n}}
\int_{\mathbb{R}^{n}}\int_{\mathbb{R}^{n}}
\frac{|f(y)|^{2}}{(1+|y-x|)^{n+1}(1+|z-x|)^{2n+2}}dy dx\|\partial^{\alpha}_{z}h(z,\cdot)\|^{2}_{L^{2}}
dz; \label{E44}\\
\|T^{*}_a(f,&\partial^{\gamma}h_{z})\|^{2}_{L^{2}_{N_{0}}}\nonumber\\
 &\lesssim  q^{2}(a)\sum\limits_{|\alpha|\leq N_{0}}
\int_{\mathbb{R}^{n}}\int_{\mathbb{R}^{n}}\int_{\mathbb{R}^{n}}
\frac{|f(y)|^{2}}{(1+|y-x|)^{n+1}(1+|z-x|)^{2n+2}}dy dx\|\partial^{\alpha}_{z}h(z,\cdot)\|^{2}_{L^{2}}
dz.\label{E45}
\end{align}
\end{lemma}

\begin{proof}
Recall
\begin{equation*}
T^{*}_a(f,h_{z})(z)=\int_{\mathbb{R}^{2n} }\int_{\mathbb{R}^{2n} } e^{ -i  (x-y)\cdot\xi} e^{ -i  (x-z)\cdot\eta} a(x,\xi,\eta) d\xi d\eta f(y)h(z,x)dydx.
 \end{equation*}
Integrate by parts with respect to variable $\xi,\eta$, and Schwartz's inequality give that
  \begin{eqnarray*}
|T^{*}_a(f,h_{z})(z)|
&\lesssim& p(a)\int_{\mathbb{R}^{n}}\int_{\mathbb{R}^{n}}\frac{|f(y)|}{(1+|y-x|)^{n+1}}
\frac{|h(z,x)|}{(1+|z-x|)^{n+1}}dy dx\\
&\lesssim& p(a)
\big(\int_{\mathbb{R}^{n}}\int_{\mathbb{R}^{n}}
\frac{|f(y)|^{2}}{(1+|y-x|)^{n+1}(1+|z-x|)^{2n+2}}dy dx\big)^{\frac{1}{2}}
\|h(z,\cdot)\|_{L^{2}},
\end{eqnarray*}
which  implies (\ref{E42}) clearly.

Let
\begin{center}
$a_{\gamma}(x,\xi,\eta)=\sum\limits_{|\beta|\leq|\gamma|}
\frac{\gamma!}{(\gamma-\beta)!\beta!}(\xi+\eta)^{\gamma-\beta}\partial^{\beta}_{x}a(x,\xi,\eta).$
\end{center}
Then integrate by parts implies that
\begin{eqnarray*}
T^{*}_a(f,\partial^{\gamma}h_{z})(z)
&=&T^{*}_{a_{\gamma}}(f,h_{z})(z)
\end{eqnarray*}
since $h_{z}$ has compact support.
Notice that
$\|\partial_{\xi}^{\alpha}\partial_{\eta}^{\beta}a_{\gamma}\|_{L^{\infty}(\mathbb{R}^{3n})}\lesssim
\|\partial_{x}^{\gamma}\partial_{\xi}^{\alpha}\partial_{\eta}^{\beta}a\|_{L^{\infty}(\mathbb{R}^{3n})}
,|\gamma|\leq N_{0}.$
By (\ref{E42}), one can get (\ref{E43}) easily.

Let $\bar{a}_{\alpha}(x,\xi,\eta)=a(x,\xi,\eta)\eta^{\alpha}$, then
$\partial^{\alpha}_{z}T^{*}_a(f,h_{z})(z)=\sum\limits_{|\beta|\leq|\alpha|}T^{*}_{\bar{a}_{\alpha-\beta}}(f,\partial^{\beta}_{z}h_{z})(z).$ Thus
\begin{eqnarray*}
\|T^{*}_a(f,h_{z})\|_{L^{2}_{N_{0}}}\lesssim \sum\limits_{|\alpha|\leq N_{0}}\sum\limits_{|\beta|\leq|\alpha|}\|T^{*}_{\bar{a}_{\alpha-\beta}}(f,\partial^{\beta}_{z}h_{z})\|_{L^{2}}.
\end{eqnarray*}
Notice that
$$|\partial^{\beta}_{\xi}\partial^{\gamma}_{\eta}\bar{a}_{\alpha}(x,\xi,\eta)|
\lesssim p(a)$$
on the set $\operatorname{supp}_{\xi,\eta} a(x,\xi,\eta)$. So (\ref{E44}) follows from (\ref{E42}). Similarly, one can get (\ref{E45}) by using (\ref{E43}).

\end{proof}

Now we return to prove (\ref{a11}). To use these lemmas, we make a dilation with respect to $x,\xi,\eta$. More precisely, set
\begin{center}
$\tilde{a}_{j}^{\nu_{1}}(x,\xi,\eta)
=a_{j}(2^{-j\varrho}x,2^{j\varrho}\xi,2^{j\varrho}\eta)\phi(\xi-\nu_{1}).$
\end{center}
It is easy to check that
\begin{eqnarray}\label{a2}
\|T^{*}_{a_{j}^{\nu_{1}}}\|_{\mathcal{L}(L^{2}\times L^{\infty},L^{2})}
&=&
\|T^{*}_{\tilde{a}_{j}^{\nu_{1}}}\|_{\mathcal{L}(L^{2}\times L^{\infty},L^{2})}.
\end{eqnarray}
Decompose $\tilde{a}_{j}^{\nu_{1}}$ with respect to $\eta$ as (\ref{eq:decomp-b}),
\begin{eqnarray}\label{E46}
T^{*}_{\tilde{a}_{j}^{\nu_{1}}}(f_{\ell},h)(z)
=\sum\limits_{\nu_{2}\in \mathbb{Z}^{n}}T^{*}_{\tilde{a}_{j}^{\nu_{1},\nu_{2}}}(f_{\ell},h)(z),
\end{eqnarray}
where
$$\tilde{a}_{j}^{\nu_{1},\nu_{2}}(x,\xi,\eta)=\tilde{a}_{j}^{\nu_{1}}(x,\xi,\eta)\phi(\eta-\nu_{2})
=a_{j}(2^{-j\varrho}x,2^{j\varrho}\xi,2^{j\varrho}\eta)\phi(\xi-\nu_{1})\phi(\eta-\nu_{2}).$$

Set
$$\bar{a}_{j}^{\nu_{1},\nu_{2}}(x,\xi,\eta)=\tilde{a}_{j}^{\nu_{1},\nu_{2}}(x,\xi+\nu_{1},\eta+\nu_{2}),$$
\begin{center}
$f_{\ell}^{\nu_{1}}(y)=e^{ -i x\cdot\nu_{1}}f_{\ell}(y),$
$h_{z}^{\nu_{1},\nu_{2}}(x)=\frac{e^{ -i x\cdot(\nu_{1}+\nu_{2})}h(x)}{(1+|x-z|^{2})^{N}}$, $h_{z}^{\nu_{1}}(x)=\frac{e^{ -i x\cdot\nu_{1}}h(x)}{(1+|x-z|^{2})^{N}},$
\end{center}
where $N>\frac{n+2}{2}$.
Then one can write
\begin{eqnarray*}
T^{*}_{\tilde{a}_{j}^{\nu_{1},\nu_{2}}}(f_{\ell},h)(z)
=e^{ i z\cdot\nu_{2}}T^{*}_{(1+\Delta_{\eta})^{N}\bar{a}_{j}^{\nu_{1},\nu_{2}}}(\phi(D)f_{\ell}^{\nu_{1}},h_{z}^{\nu_{1},\nu_{2}})(z).
\end{eqnarray*}

Substitute this into (\ref{E46}), one can get by Lemma \ref{La8}
\begin{eqnarray}\label{E47}
\|T^{*}_{\tilde{a}_{j}^{\nu_{1}}}(f,h)\|^{2}_{L^{2}}
&\lesssim&\big(\sum\limits_{\nu_{2}\in \mathbb{Z}^{n}}\|T^{*}_{(1+\Delta_{\eta})^{N}\bar{a}_{j}^{\nu_{1},\nu_{2}}}(\phi(D)f_{\ell}^{\nu_{1}},h_{z}^{\nu_{1},\nu_{2}})\|^{2}_{L^{2}}\big)^{1-t}\nonumber\\
&\times&
\big(\sum\limits_{\nu_{2}\in \mathbb{Z}^{n}}\|T^{*}_{(1+\Delta_{\eta})^{N}\bar{a}_{j}^{\nu_{1},\nu_{2}}}(\phi(D)f_{\ell}^{\nu_{1}},h_{z}^{\nu_{1},\nu_{2}})\|^{2}_{L^{2}_{N_{0}}}\big)^{t},
\end{eqnarray}
where $N_{0}=[\frac{n}{2}]+1$ and $t=\frac{n}{2N_{0}}$. Define $\lambda=2^{-jN_{0}(\delta-\varrho)}$ and
take function $u_{z}^{\nu_{1},\nu_{2}}$ such that
$$h_{z}^{\nu_{1},\nu_{2}}(x)
=u_{z}^{\nu_{1},\nu_{2}}(x)+\lambda \partial^{N_{0}}_{x}u_{z}^{\nu_{1},\nu_{2}}(x).$$
Then
(\ref{E42}) and (\ref{E43}) in Lemma \ref{La9} give
\begin{eqnarray*}
&&\|T^{*}_{(1+\Delta_{\eta})^{N}\bar{a}_{j}^{\nu_{1},\nu_{2}}}(\phi(D)f_{\ell}^{\nu_{1}},h_{z}^{\nu_{1},\nu_{2}})\|^{2}_{L^{2}}\\
&\lesssim&\|T^{*}_{(1+\Delta_{\eta})^{N}\bar{a}_{j}^{\nu_{1},\nu_{2}}}(\phi(D)f_{\ell}^{\nu_{1}},u_{z}^{\nu_{1},\nu_{2}})\|^{2}_{L^{2}}
+\lambda^{2} \|T^{*}_{(1+\Delta_{\eta})^{N}\bar{a}_{j}^{\nu_{1},\nu_{2}}}(\phi(D)f_{\ell}^{\nu_{1}},\partial^{N_{0}}_{x}u_{z}^{\nu_{1},\nu_{2}})\|^{2}_{L^{2}}\\
&\lesssim&\left(2^{2jm}
+\lambda^{2}2^{2jm+2jN_{0}(\delta-\varrho)}\right)\\
&\times&\int_{\mathbb{R}^{n}}\int_{\mathbb{R}^{n}}\int_{\mathbb{R}^{n}}
\frac{|\phi(D)f_{\ell}^{\nu_{1}}(y)|^{2}}{(1+|y-x|)^{N}(1+|z-x|)^{2N}}dy dx
\int_{\mathbb{R}^{n}}
|u_{z}^{\nu_{1},\nu_{2}}(x)|^{2} dx
dz\\
&\lesssim&2^{2jm}\int_{\mathbb{R}^{n}}\int_{\mathbb{R}^{n}}\int_{\mathbb{R}^{n}}
\frac{|\phi(D)f_{\ell}^{\nu_{1}}(y)|^{2}}{(1+|y-x|)^{N}(1+|z-x|)^{2N}}dy dx
\int_{\mathbb{R}^{n}}
|u_{z}^{\nu_{1},\nu_{2}}(x)|^{2} dx,
\end{eqnarray*}
where we used the fact that
$$\partial^{\alpha}_{x}\partial^{\beta}_{\xi}\partial^{\gamma}_{\eta}(1+\Delta_{\eta})^{N}\bar{a}_{j}^{\nu_{1},\nu_{2}}\leq C_{\alpha,\beta,\gamma,N}2^{jm+|\alpha|j(\delta-\varrho)}.$$
Notice that $\widehat{u_{z}^{\nu_{1},\nu_{2}}}(\xi)=(1+\lambda|\xi|^{N_{0}})^{-1}\widehat{h_{z}^{\nu_{1},\nu_{2}}}(\xi)$, we have
\begin{eqnarray*}
\sum\limits_{\nu_{2}\in \mathbb{Z}^{n}}\int_{\mathbb{R}^{n}}
|u_{z}^{\nu_{1},\nu_{2}}(x)|^{2} dx
&=&\sum\limits_{\nu_{2}\in \mathbb{Z}^{n}} \int_{\mathbb{R}^{n}}(1+\lambda|\xi-\nu_{2}|^{N_{0}})^{-2}|\widehat{h_{z}^{\nu_{1}}}(\xi)|^{2}d\xi\nonumber\\
&\lesssim&2^{-jn(\delta-\varrho)}\|h\|^{2}_{L^{\infty}}.
\end{eqnarray*}
Thus
\begin{eqnarray}\label{E48}
&&\sum\limits_{\nu_{2}\in \mathbb{Z}^{n}}
\|T^{*}_{(1+\Delta_{\eta})^{N}\bar{a}_{j}^{\nu_{1},\nu_{2}}}(\phi(D)f_{\ell}^{\nu_{1}},h_{z}^{\nu_{1},\nu_{2}})\|^{2}_{L^{2}}\nonumber\\
&\lesssim&2^{2jm-jn(\delta-\varrho)}\|h\|^{2}_{L^{\infty}}\int_{\mathbb{R}^{n}}\int_{\mathbb{R}^{n}}\int_{\mathbb{R}^{n}}
\frac{|\phi(D)f_{\ell}^{\nu_{1}}(y)|^{2}}{(1+|y-x|)^{N}(1+|z-x|)^{2N}}dy dx
dz\nonumber\\
&\lesssim&2^{2jm-jn(\delta-\varrho)}\|h\|^{2}_{L^{\infty}}\int_{\mathbb{R}^{n}}|\phi(D)f_{\ell}^{\nu_{1}}(y)|^{2}dy
\end{eqnarray}

Similarly, take function $u_{z,\alpha}^{\nu_{1},\nu_{2}}$ such that
$$\partial^{\alpha}_{z}h_{z}^{\nu_{1},\nu_{2}}(x)
=u_{z,\alpha}^{\nu_{1},\nu_{2}}(x)+\lambda \partial^{N_{0}}_{x}u_{z,\alpha}^{\nu_{1},\nu_{2}}(x)$$
for any $|\alpha|\leq N_{0}.$ Clearly,
\begin{eqnarray*}
\sum\limits_{\nu_{2}\in \mathbb{Z}^{n}}\int_{\mathbb{R}^{n}}
|u_{z,\alpha}^{\nu_{1},\nu_{2}}(x)|^{2} dx
&=&\sum\limits_{\nu_{2}\in \mathbb{Z}^{n}} \int_{\mathbb{R}^{n}}(1+\lambda|\xi-\nu_{2}|^{N_{0}})^{-2}|\widehat{\partial^{\alpha}_{z}h_{z}^{\nu_{1}}}(\xi)|^{2}d\xi\\
&\lesssim&2^{-jn(\delta-\varrho)}\|h\|^{2}_{L^{\infty}}.
\end{eqnarray*}
Thus (\ref{E44}) and (\ref{E45}) in Lemma \ref{La9} give
\begin{align}\label{E49}
&&\sum\limits_{\nu_{2}\in \mathbb{Z}^{n}}
\|T^{*}_{(1+\Delta_{\eta})^{N}\bar{a}_{j}^{\nu_{1},\nu_{2}}}(\phi(D)f_{\ell}^{\nu_{1}},h_{z}^{\nu_{1},\nu_{2}})\|^{2}_{L^{2}_{N_{0}}}
\lesssim  2^{2jm-jn(\delta-\varrho)}\|h\|^{2}_{L^{\infty}}\int_{\mathbb{R}^{n}}|\phi(D)f_{\ell}^{\nu_{1}}(y)|^{2}dy.
\end{align}
Substitute (\ref{E48}) and (\ref{E49}) into (\ref{E47}), one can get
\begin{eqnarray*}
&&\|T^{*}_{\tilde{a}_{j}^{\nu_{1}}}(f,h)\|^{2}_{L^{2}}\lesssim
2^{2jm-jn(\delta-\varrho)}\|h\|^{2}_{L^{\infty}}\int_{\mathbb{R}^{n}}|\phi(D)f_{\ell}^{\nu_{1}}(y)|^{2}dy.
\end{eqnarray*}
Therefore, the desired estimation (\ref{a11}) follows from (\ref{a2}) and Lemma \ref{La0} immediately.

\section{$L^{2}$ estimates: Proof of Theorem \ref{TM2}}

We only need to prove the case when $0\leq\varrho<\delta<1.$
Decompose $T_a$ as (\ref{E1}) again. Take $\epsilon_{0}$ such that
\begin{equation}\label{E39}
 \delta<\epsilon_{0}<1,
\end{equation}
and let $\psi\in \mathcal{S}(\mathbb{R}^{n})$ be that  $\hat{\psi}(\xi)=1$ if $|\xi|\leq 2^{-4}$ and $\hat{\psi}(\xi)$ if $|\xi|>2^{-3}.$  Then one can write
$$a_{j}=p_{j}+q_{j}$$
with
$$p_{j}(x,\xi,\eta)=\int_{\mathbb{R}^{n}}a_{j}(x-u,\xi,\eta)\psi(2^{j\epsilon_{0}}u)2^{j\epsilon_{0} n}du$$
$$q_{j}(x,\xi,\eta)=\int_{\mathbb{R}^{n}}\big(a_{j}(x,\xi,\eta)-a_{j}(x-u,\xi,\eta)\big)\psi(2^{j\epsilon_{0}}u)2^{j\epsilon_{0} n}du.$$
Moreover,
$$|\partial^{\alpha}_{x}\partial^{\beta}_{\xi}\partial^{\gamma}_{\eta}p_{j}(x,\xi,\eta)|\leq C_{\alpha,\beta,\gamma}2^{jm+j\delta|\alpha| -j\varrho(|\beta|+|\gamma|)},$$
$$|\partial^{\alpha}_{x}\partial^{\beta}_{\xi}\partial^{\gamma}_{\eta}q_{j}(x,\xi,\eta)|\leq C_{\alpha,\beta,\gamma}2^{jm-(\epsilon_{0}-\delta)+j\delta|\alpha| -j\varrho(|\beta|+|\gamma|)}$$
for any multi-indices $\alpha,\beta,\gamma\in \mathbb{N}^{n}.$ Next we estimate for $T_{p_j}$ and $T_{q_{j}}$ respectively.

\subsection{Estimates for $T_{p_j}$}
As the proof of Theorem \ref{TM}, we apply Schur's lemma to estimate for trilinear form:
\[
\sum_{j=1}^{\infty}\langle T_{p_{j}}(f,g), h \rangle.
\]
However, $\widehat{T_{p_{j}}(f,g)}(\zeta)$ supports on a ball instead of an annular area.
In fact
\begin{center}
$\operatorname{supp}\widehat{T_{p_{j}}(f,g)}\subset\bigcup\limits_{|\xi|+|\eta|\approx2^{j}}\{\zeta:|\zeta-\xi-\eta|\leq 2^{j\epsilon_{0}+1}\}\subset \{\zeta \in \mathbb{R}^n : |\zeta| \leq 2^{j+3}\}.$
\end{center}
Fortunately, the author in \cite{MiyachiTomita2020} find that either \( \xi + \eta \neq 0 \) or \( \xi \neq 0 \) if \( (\xi, \eta) \) belongs to the unit sphere \( \Sigma\) of \( \mathbb{R}^n \times \mathbb{R}^n \). This implies that $\widehat{T_{p_{j}}(f,g)}$ or $\hat{f}$ (can been see as a function) support on an annular area. Actually, by the compactness of \( \Sigma\), there exists a constant \( c > 0 \) such that \( \Sigma\) is covered by the two open sets
\[
V_1 = \{ (\xi, \eta) \in \Sigma: |\xi + \eta| > c \}, \quad V_2 = \{ (\xi, \eta) \in \Sigma: |\xi| > c \}.
\]
Taking a smooth partition of unity \( \Phi_i, i = 1, 2 \), on \( \Sigma\) such that \(\text{supp} \Phi_i \subset V_i \), we decompose \( p_{j} \) as
\[
p_{j} = \sum_{i=1}^{2} p_{j} \Phi_i(\zeta / |\zeta|):={}^{1}p_{j}+{}^{2}p_{j}, \quad \zeta = (\xi, \eta).
\]
Obviously \( {}^{i}p_{j} \in B S^{m}_{\varrho,\delta}(\mathbb{R}^n) \). Moreover, \( {}^{1}p_{j}\) and \( {}^{2}p_{j}\) satisfy the additional conditions
\[
 \operatorname{supp}_{\xi,\eta} {}^{1}p_{j} \subset\{(\xi, \eta):|\xi + \eta|\approx(|\xi|^2 + |\eta|^2)^{1/2} \approx 2^{j}\}
\]
and
\[
 \operatorname{supp}_{\xi,\eta} {}^{2}p_{j} \subset\{(\xi, \eta):|\xi|\approx(|\xi|^2 + |\eta|^2)^{1/2}\approx 2^{j}\},
\]
respectively. Thus we can write
\begin{equation}\label{E13}
\langle T_{{}^{1}p_{j}}(f,g), h \rangle = \langle T_{{}^{1}p_{j}}(f,g), h_j \rangle, \quad h_j = \theta(2^{-j}D)h,
\end{equation}
and
\begin{equation}\label{E14}
\langle T_{{}^{2}p_{j}}(f,g), h \rangle = \langle T_{{}^{2}p_{j}}(f_j,g), h \rangle, \quad f_j = \theta(2^{-j}D)f,
\end{equation}
where \(\theta\) is an appropriate function supported in an annulus.
Decomposing $f$ in (\ref{E13}) and $h$ in (\ref{E14}) as
\[
f = \sum_{\ell} f_\ell  ~ {\rm with} ~ f_\ell = \psi_\ell(D)f
\]
and
\[
h = \sum_{\ell} h_{\ell}~  ~{\rm with} ~ h_{\ell} = \psi_{\ell}(D)h,
\]
respectively. Then both \(\langle T_{{}^{1}p_{j}}(f_{\ell},g), h_{j} \rangle \) and \(\langle T_{{}^{2}p_{j}}(f_j,g), h_{\ell} \rangle\) equal to zero for \(\ell > j+1\).
Next we claim that
\begin{equation}\label{E30}
|\langle T_{{}^{1}p_{j}}(f_{\ell},g), h_{j} \rangle|
\lesssim (\max\{1, 2^{\ell-j\varrho}\})^n2^{-n(1-\varrho)}\|f_{\ell}\|^{2}_{L^{2}}\|h_{j}\|^{2}_{L^{2}}\|g\|^{2}_{L^{\infty}}
\end{equation}
and
\begin{equation}\label{E31}
|\langle T_{{}^{2}p_{j}}(f_j,g), h_{\ell} \rangle|
\lesssim (\max\{2^{j(\epsilon_{0}-\varrho)}, 2^{\ell-j\varrho}\})^n2^{-n(1-\varrho)}\|h_{\ell}\|^{2}_{L^{2}}\|f_{j}\|^{2}_{L^{2}}\|g\|^{2}_{L^{\infty}},
\end{equation}
where $\epsilon_{0}$ is given by (\ref{E39}). Clearly, the proof can be finished by applying Schur's lemma to (\ref{E30}) and (\ref{E31}).

\begin{lemma}\label{La10}
Let $a(x,\xi,\eta)$ be a bounded continuous function with bounded continuous derivatives $\partial_{x}^{\gamma}\partial_{\xi}^{\alpha}\partial_{\eta}^{\beta}a(x,\xi,\eta)$ for $|\alpha|,|\beta|\leq n+1 $, $|\gamma|\leq N_{0}=:[\frac{n}{2}]+1$, and assume
\[
 \operatorname{supp}_{\xi,\eta} a(x,\xi,\eta) \subset\{(\xi, \eta):|\xi|\leq n , |\eta|\leq n\}.
\]
Define
\begin{center}
$p(a)=\sum\limits_{|\alpha|,|\beta|\leq n+1}\|\partial_{\xi}^{\alpha}\partial_{\eta}^{\beta}a\|_{L^{\infty}(\mathbb{R}^{3n})},$
$q(a)=\sum\limits_{|\alpha|,|\beta|\leq n+1,|\gamma|\leq N_{0}}\|\partial_{x}^{\gamma}\partial_{\xi}^{\alpha}\partial_{\eta}^{\beta}a\|_{L^{\infty}(\mathbb{R}^{3n})}.$
\end{center}
Then the operator $T_a$ defined by (\ref{Df}) satisfying the following estimates:
  \begin{eqnarray}\label{E32}
\|T_{a}(f,g)\|^{2}_{L^{2}}
\lesssim p^{2}(a)\int_{\mathbb{R}^{n}}\int_{\mathbb{R}^{n}}\frac{|f(y)|^{2}}{(1+|y-x|)^{n+1}}dy
\int_{\mathbb{R}^{n}}\frac{|g(z)|^{2}}{(1+|z-x|)^{n+1}}dzdx
\end{eqnarray}
and
  \begin{eqnarray}\label{E33}
\|T_{a}(f,g)\|^{2}_{L^{2}_{N_{0}}}
&\lesssim&q^{2}(a)\int_{\mathbb{R}^{n}}\int_{\mathbb{R}^{n}}\frac{|f(y)|^{2}}{(1+|y-x|)^{n+1}}dy
\int_{\mathbb{R}^{n}}\frac{|g(z)|^{2}}{(1+|z-x|)^{n+1}}dzdx
\end{eqnarray}
\end{lemma}

\begin{proof}
Integrate by parts and Schwartz's inequality give that
 \begin{eqnarray*}
|T_{a}(f,g)(x)|
&\lesssim&p(a)\int_{\mathbb{R}^{n}}\frac{|f(y)|^{2}}{(1+|y-x|)^{n+1}}dy
\int_{\mathbb{R}^{n}}\frac{|g(z)|^{2}}{(1+|z-x|)^{n+1}}dz,
\end{eqnarray*}
which implies (\ref{E32}) clearly.

Let $a^{\alpha}(x,\xi,\eta)=\sum\limits_{|\beta|\leq|\alpha|}\frac{\alpha!}{(\alpha-\beta)!\beta!}\partial^{\beta}_{x}a(x,\xi,\eta)(\xi+\eta)^{\alpha-\beta}$. Then
\begin{eqnarray*}
\partial^{\alpha}_{x}T_{a}(f,g)=T_{a^{\alpha}}(f,g).
\end{eqnarray*}
Notice that
\begin{eqnarray*}
\|T_{a}(f,g)\|_{L^{2}_{N_{0}}}\lesssim \sum\limits_{|\alpha|\leq N_{0}}\|\partial^{\alpha}_{x}T_{a}(f,g)\|_{L^{2}},
\end{eqnarray*}
and
$$|\partial^{\alpha}_{x}\partial^{\beta}_{\xi}\partial^{\gamma}_{\eta}a^{\alpha}(x,\xi,\eta)|
\lesssim q(a)$$
on the set $\operatorname{supp}_{\xi,\eta} a(x,\xi,\eta)$, one can get (\ref{E32}) by using (\ref{E33}).
\end{proof}

\subsubsection{The proof of inequality (\ref{E30})}
It suffices to show
\begin{equation}\label{E35}
\|T_{{}^{1}p_{j}}(f_{\ell},h)\|^{2}_{L^{2}}
\lesssim(\max\{1, 2^{\ell-j\varrho}\})^n2^{-n(1-\varrho)}\|f_{\ell}\|^{2}_{L^{2}}\|h\|^{2}_{L^{\infty}}.
\end{equation}
To this end, we decompose the symbol ${}^{1}p_{j}(x,\xi,\eta)$ with respect to variable $\xi$ as in (\ref{eq:decomp-a})
\begin{equation}\label{E34}
    T_{{}^{1}p_{j}} = \sum\limits_{\nu_{1}\in \mathbb{Z}^{n}} T_{{}^{1}p_{j}^{\nu_{1}}}
\end{equation}
with
\begin{eqnarray}
{}^{1}p_{j}^{\nu_{1}}(x,\xi,\eta)
={}^{1}p_{j}(x,\xi,\eta) \phi(2^{-j\varrho}\xi-\nu_{1}).
\end{eqnarray}
Define
\[
\Lambda_{j,\ell} = \{\nu_1 \in \mathbb{Z}^n : \operatorname{supp} \varphi(2^{-j\varrho} - \nu_1) \cap \operatorname{supp} \psi_\ell \neq \emptyset\}.
\]
Clearly, the number of elements of \(\Lambda_{j,\ell}\) satisfies
\[
|\Lambda_{j,\ell}| \lesssim (\max\{1, 2^{\ell-j\varrho}\})^n.
\]
Thus, (\ref{E34}) and Schwartz's inequality give
\begin{equation}\label{E10}
\|T_{{}^{1}p_{j}}(f_{\ell},h)\|^{2}_{L^{2}}
\leq|\Lambda_{j,\ell}|\sum\limits_{\nu_{1}\in \Lambda_{j,\ell}}\|T_{{}^{1}p_{j}^{\nu_{1}}}(f_{\ell},h)\|^{2}_{L^{2}}.
\end{equation}
Set
\begin{center}
${}^{1}\tilde{p}_{j}^{\nu_{1}}(x,\xi,\eta)={}^{1}p_{j}^{\nu_{1}}(2^{-j\varrho}x,2^{j\varrho}\xi,2^{j\varrho}\eta).$
\end{center}
It is easy to check that
\begin{eqnarray}\label{E9}
\|T_{{}^{1}p_{j}^{\nu_{1}}}\|_{\mathcal{L}(L^{2}\times L^{\infty},L^{2})}
&=&
\|T_{{}^{1}\tilde{p}_{j}^{\nu_{1}}}\|_{\mathcal{L}(L^{2}\times L^{\infty},L^{2})}.
\end{eqnarray}

Decompose ${}^{1}\tilde{p}_{j}^{\nu_{1}}$ with respect to variable $\eta$ as (\ref{eq:decomp-b}),
\begin{eqnarray*}
T_{{}^{1}\tilde{p}_{j}^{\nu_{1}}}(f_{\ell},h)(z)
=\sum\limits_{\nu_{2}\in \mathbb{Z}^{n}}T_{{}^{1}\tilde{p}_{j}^{\nu_{1},\nu_{2}}}(f_{\ell},h)(z)
\end{eqnarray*}
with
$${}^{1}\tilde{p}_{j}^{\nu_{1},\nu_{2}}(x,\xi,\eta)={}^{1}\tilde{p}_{j}^{\nu_{1}}(x,\xi,\eta)\phi(\eta-\nu_{2}).$$
Set

\begin{center}
${}^{1}\bar{p}_{j}^{\nu_{1},\nu_{2}}(x,\xi,\eta)={}^{1}\tilde{p}_{j}^{\nu_{1},\nu_{2}}(x,\xi+\nu_{1},\eta+\nu_{2}),$
$f_{\ell}^{\nu_{1}}(y)=e^{ -i x\cdot\nu_{1}}f_{\ell}(y),$
$h^{\nu_{2}}(x)=e^{ -i x\cdot\nu_{2}}h(x).$
\end{center}
Then we can write
\begin{eqnarray*}
T_{{}^{1}\tilde{p}_{j}^{\nu_{1}}}(f_{\ell},h)(z)
=e^{ i x\cdot\nu_{1}}\sum\limits_{\nu_{2}\in \mathbb{Z}^{n}}e^{ i x\cdot\nu_{2}}T_{{}^{1}\bar{p}_{j}^{\nu_{1},\nu_{2}}}(\phi(D)f_{\ell}^{\nu_{1}},\phi(D)h^{\nu_{2}})(z).
\end{eqnarray*}
Moreover, Lemma \ref{La9} gives
\begin{eqnarray}\label{E36}
\|T_{{}^{1}\tilde{p}_{j}^{\nu_{1}}}(f_{\ell},h)\|^{2}_{L^{2}}
&\lesssim&\big(\sum\limits_{\nu_{2}\in \mathbb{Z}^{n}}\|T_{{}^{1}\bar{p}_{j}^{\nu_{1},\nu_{2}}}(\phi(D)f_{\ell}^{\nu_{1}},\phi(D)h^{\nu_{2}})\|^{2}_{L^{2}}\big)^{1-t}\nonumber\\
&\times&
\big(\sum\limits_{\nu_{2}\in \mathbb{Z}^{n}}\|T_{{}^{1}\bar{p}_{j}^{\nu_{1},\nu_{2}}}(\phi(D)f_{\ell}^{\nu_{1}},\phi(D)h^{\nu_{2}})\|^{2}_{L^{2}_{N_{0}}}\big)^{t}.
\end{eqnarray}
Notice that
\[
 \operatorname{supp}_{\xi,\eta} {}^{1}\bar{p}_{j}^{\nu_{1},\nu_{2}}(x,\xi,\eta) \subset\{(\xi, \eta):|\xi|\leq n , |\eta|\leq n\}
\]
and
\begin{center}
$|\partial^{\alpha}_{x}\partial^{\beta}_{\xi}\partial^{\gamma}_{\eta}{}^{1}\bar{p}_{j}^{\nu_{1},\nu_{2}}(x,\xi,\eta)|
\leq C_{\alpha,\beta,\gamma}2^{jm+j|\alpha|(\delta-\varrho)},
\quad \forall\alpha,\beta\in \mathbb{N}^{n}.$
\end{center}
Thus, by Lemma \ref{La10} and Lemma \ref{La0}
\begin{eqnarray}\label{E37}
&&\sum\limits_{\nu_{2}\in \mathbb{Z}^{n}}
\|T_{{}^{1}\bar{p}_{j}^{\nu_{1},\nu_{2}}}(\phi(D)f_{\ell}^{\nu_{1}},\phi(D)h^{\nu_{2}})\|^{2}_{L^{2}}\nonumber\\
&\lesssim& 2^{2jm}\int_{\mathbb{R}^{n}}\int_{\mathbb{R}^{n}}\frac{|\phi(D)f_{\ell}^{\nu_{1}}(y)|^{2}}{(1+|y-x|)^{n+1}}dy
\int_{\mathbb{R}^{n}}\frac{\sum\limits_{\nu_{2}\in \mathbb{Z}^{n}}|\phi(D)h^{\nu_{2}}(z)|^{2}}{(1+|z-x|)^{n+1}}dzdx\nonumber\\
&\lesssim& 2^{2jm}\|h\|^{2}_{L^{\infty}}\int_{\mathbb{R}^{n}}|\phi(D)f_{\ell}^{\nu_{1}}(y)|^{2}dy.
\end{eqnarray}
Similarly,
\begin{equation}\label{E38}
\sum\limits_{\nu_{2}\in \mathbb{Z}^{n}}\|T_{{}^{1}\bar{p}_{j}^{\nu_{1},\nu_{2}}}(\phi(D)f_{\ell}^{\nu_{1}},\phi(D)h^{\nu_{2}})\|^{2}_{L^{2}_{N_{0}}}
\lesssim2^{2jm+2N_{0}(\delta-\varrho)}\|h\|^{2}_{L^{\infty}}\int_{\mathbb{R}^{n}}|\phi(D)f_{\ell}^{\nu_{1}}(y)|^{2}dy.
\end{equation}
Recall $t=\frac{n}{2N_{0}}$ and $m = -\frac{n}{2}(1-\varrho)-\frac{n}{2}(\delta-\varrho)$. Combining (\ref{E37}),(\ref{E38}) and (\ref{E36}), we arrive at
\begin{eqnarray*}
\|T_{{}^{1}\tilde{p}_{j}^{\nu_{1}}}(f_{\ell},h)\|^{2}_{L^{2}}
\lesssim2^{-n(1-\varrho)}\|h\|^{2}_{L^{\infty}}\int_{\mathbb{R}^{n}}|\phi(D)f_{\ell}^{\nu_{1}}(y)|^{2}dy
\end{eqnarray*}
Thus, the desired estimate (\ref{E35}) follows from (\ref{E10}), (\ref{E9}) and Lemma \ref{La0}.

\subsubsection{The proof of inequality (\ref{E31})}
The proof for this part is similar to that of the previous part. We decompose the operators $T_{{}^{2}p_{j}}$ as in (\ref{eq:decomp-c})
\begin{equation}
    T_{{}^{2}p_{j}} = \sum\limits_{\nu_{1}\in \mathbb{Z}^{n}} \sum\limits_{\nu_{2}\in \mathbb{Z}^{n}}T_{{}^{2}p_{j}^{\nu_{1},\nu_{2}}}.
\end{equation}
Observing that
\[
\langle T_{{}^{2}p_{j}^{\nu_{1},\nu_{2}}}(f_j,g), h_{\ell} \rangle \neq 0 \implies (2^{j\varrho}(\nu_1 + \nu_2) + 2^{j\epsilon_{0}}Q) \cap \operatorname{supp} \psi_{\ell} \neq \emptyset,
\]
one can write
\[
\langle T_{{}^{2}p_{j}}(f_j,g), h_{\ell} \rangle= \sum\limits_{\nu_{1}\in \mathbb{Z}^{n}}\sum\limits_{\nu_{2}\in \{\nu-\nu_{1}:\nu\in\Lambda_{j,\ell}\}} \langle T_{{}^{2}p_{j}^{\nu_{1},\nu_{2}}}(f_j,g), h_{\ell}\rangle,
\]
where
\[
\Lambda_{j,\ell} = \{\nu \in \mathbb{Z}^n : (2^{j\varrho} \nu + 2^{j\epsilon_{0}}Q) \cap \operatorname{supp} \psi_{\ell} \neq \emptyset\}.
\]
Thus, Schwartz's inequality gives that
\begin{eqnarray*}
\left|\langle T_{{}^{2}p_{j}}(f_j,g), h_{\ell} \rangle \right|
&\leq&\left\| \sum\limits_{\nu_{1}\in \mathbb{Z}^{n}}\sum\limits_{\nu_{2}\in \{\nu-\nu_{1}:\nu\in\Lambda_{j,\ell}\}} T_{{}^{2}p_{j}^{\nu_{1},\nu_{2}}}(f_j, g) \right\|_{L^2} \|h_\ell\|_{L^2}.
\end{eqnarray*}
Next we claim that
\begin{eqnarray}\label{E40}
\| \sum\limits_{\nu_{1}\in \mathbb{Z}^{n}}\sum\limits_{\nu_{2}\in \{\nu-\nu_{1}:\nu\in\Lambda_{j,\ell}\}} T_{{}^{2}p_{j}^{\nu_{1},\nu_{2}}}(f_j, g) \|^{2}_{L^2}
\lesssim2^{-jn(1-\varrho)}|\Lambda_{j,\ell}|\|g\|^{2}_{L^{\infty}}\|f_{j}\|^{2}_{L^{2}}.
\end{eqnarray}
Then the proof of this part can be finished by the fact that the number of elements of \(\Lambda_{j,\ell}\) is estimated by
\[
|\Lambda_{j,\ell}| \lesssim (\max\{2^{j(\epsilon_{0}-\varrho)}, 2^{\ell-j\varrho}\})^n.
\]

Set
$${}^{2}\tilde{p}_{j}^{\nu_{1},\nu_{2}}(x,\xi,\eta)
={}^{2}p_{j}(2^{-j\varrho}x,2^{j\varrho}\xi,2^{j\varrho}\eta)\phi(\xi-\nu_{1})\phi(\eta-\nu_{2}).$$
It is easy to check that
\begin{eqnarray}\label{E51}
&&\|\sum\limits_{\nu_{1}\in \mathbb{Z}^{n}}\sum\limits_{\nu_{2}\in \{\nu-\nu_{1}:\nu\in\Lambda_{j,\ell}\}} T_{{}^{2}p_{j}^{\nu_{1},\nu_{2}}}(f_j, g)\|_{\mathcal{L}(L^{2}\times L^{\infty},L^{2})}\nonumber\\
&=&
\|\sum\limits_{\nu_{1}\in \mathbb{Z}^{n}}\sum\limits_{\nu_{2}\in \{\nu-\nu_{1}:\nu\in\Lambda_{j,\ell}\}} T_{{}^{2}\tilde{p}_{j}^{\nu_{1},\nu_{2}}}(f_j, g)\|_{\mathcal{L}(L^{2}\times L^{\infty},L^{2})}.
\end{eqnarray}
Let
\begin{center}
${}^{2}\bar{p}_{j}^{\nu_{1},\nu_{2}}(x,\xi,\eta)={}^{2}\tilde{p}_{j}^{\nu_{1},\nu_{2}}(x,\xi+\nu_{1},\eta+\nu_{2}),$
$f_{j}^{\nu_{1}}(x)=e^{ -i x\cdot\nu_{1}}f(x),g_{\nu_{2}}(x)=e^{ -i x\cdot\nu_{2}}g(x)$.
\end{center}
Then
\begin{eqnarray*}
&&\sum\limits_{\nu_{1}\in \mathbb{Z}^{n}}\sum\limits_{\nu_{2}\in \{\nu-\nu_{1}:\nu\in\Lambda_{j,\ell}\}} T_{{}^{2}\tilde{p}_{j}^{\nu_{1},\nu_{2}}}(f_j, g)(x)\\
&&=\sum\limits_{\nu_{1}\in \mathbb{Z}^{n}}e^{ i x\cdot\nu_{1}}
\sum\limits_{\nu_{2}\in \{\nu-\nu_{1}:\nu\in\Lambda_{j,\ell}\}}e^{ i x\cdot\nu_{2}} T_{{}^{2}\tilde{p}_{j}^{\nu_{1},\nu_{2}}}(\phi(D)f_{j}^{\nu_{1}}, \phi(D)g_{\nu_{2}})(x).
\end{eqnarray*}
Moreover, Lemma \ref{La9} and Schwartz's inequality imply that
\begin{eqnarray}\label{E50}
&&\|\sum\limits_{\nu_{1}\in \mathbb{Z}^{n}}e^{ i x\cdot\nu_{1}}
\sum\limits_{\nu_{2}\in \{\nu-\nu_{1}:\nu\in\Lambda_{j,\ell}\}}e^{ i x\cdot\nu_{2}} T_{{}^{2}\tilde{p}_{j}^{\nu_{1},\nu_{2}}}(\phi(D)f_{j}^{\nu_{1}}, \phi(D)g_{\nu_{2}})\|^{2}_{L^{2}}\nonumber\\
&\lesssim&\big(\sum\limits_{\nu_{1}\in \mathbb{Z}^{n}}\|\sum\limits_{\nu_{2}\in \{\nu-\nu_{1}:\nu\in\Lambda_{j,\ell}\}}e^{ i x\cdot\nu_{2}}T_{{}^{2}\tilde{p}_{j}^{\nu_{1},\nu_{2}}}(\phi(D)f_{j}^{\nu_{1}}, \phi(D)g_{\nu_{2}})\|^{2}_{L^{2}}\big)^{1-t}\nonumber\\
&\times&\big(\sum\limits_{\nu_{1}\in \mathbb{Z}^{n}}\|\sum\limits_{\nu_{2}\in \{\nu-\nu_{1}:\nu\in\Lambda_{j,\ell}\}}e^{ i x\cdot\nu_{2}}T_{{}^{2}\tilde{p}_{j}^{\nu_{1},\nu_{2}}}(\phi(D)f_{j}^{\nu_{1}}, \phi(D)g_{\nu_{2}})\|^{2}_{L^{2}_{N_{0}}}\big)^{t}\nonumber\\
&\lesssim&|\Lambda_{j,\ell}|\big(\sum\limits_{\nu_{1}\in \mathbb{Z}^{n}}\sum\limits_{\nu_{2}\in \mathbb{Z}^{n}}\|T_{{}^{2}\tilde{p}_{j}^{\nu_{1},\nu_{2}}}(\phi(D)f_{j}^{\nu_{1}}, \phi(D)g_{\nu_{2}})\|^{2}_{L^{2}}\big)^{1-t}\nonumber\\
&\times&\big(\sum\limits_{\nu_{1}\in \mathbb{Z}^{n}}\sum\limits_{\nu_{2}\in \mathbb{Z}^{n}}\|T_{{}^{2}\tilde{p}_{j}^{\nu_{1},\nu_{2}}}(\phi(D)f_{j}^{\nu_{1}}, \phi(D)g_{\nu_{2}})\|^{2}_{L^{2}_{N_{0}}}\big)^{t}.
\end{eqnarray}

From Lemma \ref{La10} and Lemma \ref{La0}, it follows that
\begin{eqnarray*}
\sum\limits_{\nu_{1}\in \mathbb{Z}^{n}}\sum\limits_{\nu_{2}\in \mathbb{Z}^{n}}
\|T_{\bar{p}_{j}^{\nu_{1},\nu_{2}}}(\phi(D)f_{j}^{\nu_{1}},\phi(D)g_{\nu_{2}})\|^{2}_{L^{2}}
\lesssim 2^{2jm}\|g\|^{2}_{L^{\infty}}\|f_{j}\|^{2}_{L^{2}}
\end{eqnarray*}
and
\begin{eqnarray*}
\sum\limits_{\nu_{2}\in \mathbb{Z}^{n}}\|T_{\bar{p}_{j}^{\nu_{1},\nu_{2}}}(\phi(D)f_{j}^{\nu_{1}},\phi(D)h^{\nu_{2}})\|^{2}_{L^{2}_{N_{0}}}
&\lesssim&2^{2jm+2N_{0}(\delta-\varrho)}\|g\|^{2}_{L^{\infty}}\|f_{j}\|^{2}_{L^{2}}.
\end{eqnarray*}
Recall $t=\frac{n}{2N_{0}}$ and $m = -\frac{n}{2}(1-\varrho)-\frac{n}{2}(\delta-\varrho)$. Substitute these two estimates into (\ref{E50}), one can get the desired estimation (\ref{E40}) by (\ref{E51}).

\subsection{Estimates for $T_{q_j}$}
It suffices to show
\begin{equation}\label{E53}
\|T_{q_{j}}(f,g)\|_{L^{2}}\lesssim 2^{-j(\epsilon_{0}-\delta)}\|f\|_{L^2}\|g\|_{L^{\infty}}
\end{equation}
which be proved in parallel with the estimate for $T_{{}^{1}p_{j}}$, that is (\ref{E35}) above. We merely present the outline.

Decompose the symbol $q_{j}(x,\xi,\eta)$ as in (\ref{eq:decomp-a})
\begin{center}
$
T_{q_{j}} = \sum\limits_{\nu_{1}\in \mathbb{Z}^{n}} T_{q_{j}^{\nu_{1}}}
$
with
$
q_{j}^{\nu_{1}}(x,\xi,\eta)
=q_{j}(x,\xi,\eta) \phi(2^{-j\varrho}\xi-\nu_{1}).
$
\end{center}
Define
\[
\Lambda_{j} = \{\nu_1 \in \mathbb{Z}^n : \operatorname{supp} \varphi(2^{-j\varrho}\cdot - \nu_1) \cap \operatorname{supp_{\xi}} p_{j} \neq \emptyset\}.
\]
Then
\[
|\Lambda_{j}| \lesssim 2^{j(1-\varrho)n}.
\]
Thus,
\begin{equation}\label{E52}
\|T_{q_{j}}(f,h)\|^{2}_{L^{2}}
\lesssim2^{j(1-\varrho)n}\sum\limits_{\nu_{1}\in \Lambda_{j}}\|T_{q_{j}^{\nu_{1}}}(f,h)\|^{2}_{L^{2}}.
\end{equation}
By making a dilation
\begin{center}
$\tilde{q}_{j}^{\nu_{1}}(x,\xi,\eta)=q_{j}^{\nu_{1}}(2^{-j\varrho}x,2^{j\varrho}\xi,2^{j\varrho}\eta)$
\end{center}
and the same argument as above, we can get
\begin{eqnarray*}
\|T_{\tilde{q}_{j}^{\nu_{1}}}(f,h)\|^{2}_{L^{2}}
\lesssim2^{-n(1-\varrho)-j(\epsilon_{0}-\delta)}\|h\|^{2}_{L^{\infty}}\int_{\mathbb{R}^{n}}|\phi(D)f_{\nu_{1}}(y)|^{2}dy,
\end{eqnarray*}
where we used the crucial fact that
\begin{center}
$|\partial^{\alpha}_{x}\partial^{\beta}_{\xi}\partial^{\gamma}_{\eta}q_{j}(x,\xi,\eta)|\leq C_{\alpha,\beta,\gamma}2^{jm-j(\epsilon_{0}-\delta)+j(\delta-\varrho)|\alpha|},
\quad \forall \alpha,\beta,\gamma\in \mathbb{N}^{n}.$
\end{center}
Thus, one can get the desired estimate (\ref{E53}) by (\ref{E52}) and Lemma \ref{La0}.

\section{Hardy space estimates: Proof of Theorems \ref{T2} and \ref{T4}}
Both Theorem \ref{T2} and Theorem \ref{T4} in the case when $0\leq\delta\leq\varrho<1$ have been proved in \cite{MiyachiTomita2018}.
The proof methods of these two theorems in the case when $0\leq\varrho<\delta<1$ are almost the same as those of Theorem 1.2 and Theorem 1.3 in \cite{MiyachiTomita2018}, except for the method of estimating the kernels of the bilinear pseudo-differential operator. For convenience, we present the proof framework of both theorems in Appendix A, and we merely list some key steps of the estimate for the kernels in this section.

After applying certain decompositions to the symbol $a$ of the bilinear operator $T_{a}$, the authors in \cite{MiyachiTomita2018} used two important theorems on $\mathbb{R}^{2n}$: the Calder\'{o}n-Vaillancourt theorem and Plancherel's theorem to estimate the $L^{2}_{x,z}$-norm of the corresponding kernel. For the case when $0\leq\varrho<\delta<1$, the same estimate still holds if one uses Hounie's theorem \cite{Hounie2} instead of the Calder\'{o}n-Vaillancourt theorem. However, one has to assume that the order $m\leq-\frac{n(1-\varrho)}{p}-n\max\{\delta-\varrho,0\}$. To improve the order to $m\leq-\frac{n(1-\varrho)}{p}-\frac{n}{2}\max\{\delta-\varrho,0\}$, we employ the main idea of Hounie's \cite{Hounie2} again and Plancherel's theorem on $\mathbb{R}^{n}$ to estimate the $L^{2}_{x,z}$-norm of these kernels. More precisely, we shall prove the following estimates:
\begin{eqnarray}\label{E26}
\left\|(2^{j\varrho}(x-[c_Q,y]_t))^{\beta}(2^{j\varrho}z)^{\gamma}K_{j,\ell}^{(\alpha,0)}(x,x-[c_Q,y]_t,z)\right\|_{L^2_{x,z}}
\lesssim 2^{j(|\alpha|-\frac{n}{p}(1-\varrho)+\frac{n}{2})}2^{(j\varrho+\ell)\frac{n}{2}}
\end{eqnarray}
and
\begin{eqnarray}\label{E27}
\left\| (2^{j\varrho}(x-ty))^{\beta} (2^{j\varrho}z)^{\gamma} K_j^{(\alpha,0)}(x,x-ty,z) \right\|_{L^2_{x,z}}
\lesssim 2^{j(|\alpha|-\frac{n}{p}(1-\varrho)+n)},
\end{eqnarray}
on page 9 and page 15 in \cite{MiyachiTomita2018}, respectively.
We shall prove only (\ref{E26}), since (\ref{E27}) can be obtained by a similar argument.

Recall that
\begin{eqnarray*}
&(2^{j\varrho}(x-[c_Q,y]_t))^{\beta}(2^{j\varrho}z)^{\gamma}K_{j,\ell}^{(\alpha,0)}(x,x-[c_Q,y]_t,z)\\[2mm]
&= C_{\alpha,\beta,\gamma}\int_{(\mathbb{R}^n)^2} e^{i\{x\cdot\xi+z\cdot\eta\}}(2^{j\varrho}\partial_{\xi})^{\beta}(2^{j\varrho}\partial_{\eta})^{\gamma}\left[\xi^{\alpha}a_{j,\ell}(x,\xi,\eta)\right] \\[1mm]
&\quad\times e^{-i[c_Q,y]_t\cdot\xi}\psi_0(\xi/2^{j+1})\psi_0(\eta/2^{j\varrho+\ell+1})\,d\xi\,d\eta,
\end{eqnarray*}
where the symbol satisfies
\[
\begin{aligned}
&\operatorname{supp} a_{j,\ell}(x,\cdot,\cdot) \subset \{|\xi| \le 2^{j+1},\ |\eta| \le 2^{j\varrho+\ell+1}\}, \\[1mm]
&1 + |\xi| + |\eta| \approx 2^j \quad \text{on } \operatorname{supp} a_{j,\ell}(x,\cdot,\cdot), \\[1mm]
&|\partial_x^\alpha \partial_\xi^\beta \partial_\eta^\gamma a_{j,\ell}(x,\xi,\eta)| \lesssim 2^{-j\frac{n(1-\varrho)}{p}-j\frac{n}{2}\max\{\delta-\varrho,0\}} 2^{j\varrho(|\alpha|-|\beta|-|\gamma|)}.
\end{aligned}
\]

Applying Plancherel's theorem with respect to the variable $z\rightarrow\eta$, one sees that the left-hand side of (\ref{E26}) equals
\begin{eqnarray}\label{E28}
&C_{\alpha,\beta,\gamma}&\int_{\mathbb{R}^{n}}\int_{\mathbb{R}^{n}}|
\int_{\mathbb{R}^{n}} e^{i\{x\cdot\xi\}}(2^{j\varrho}\partial_{\xi})^{\beta}(2^{j\varrho}\partial_{\eta})^{\gamma}\left[\xi^{\alpha}a_{j,\ell}(x,\xi,\eta)\right] e^{-i[c_Q,y]_t\cdot\xi}\psi_0(\xi/2^{j+1})d\xi|^{2}dx\nonumber\\[2mm]
&\times&\psi^{2}_0(\eta/2^{j\varrho+\ell+1}) d\eta.
\end{eqnarray}
Set
\begin{center}
$A^{\eta}_{j,\ell}(x,\xi)=(2^{j\varrho}\partial_{\xi})^{\beta}(2^{j\varrho}\partial_{\eta})^{\gamma}\xi^{\alpha}a_{j,\ell}(x,\xi,\eta)$
 and
$\hat{H}(\xi)=e^{-i[c_Q,y]_t\cdot\xi}\psi_0(\xi/2^{j+1})$.
\end{center}
Then one can write
$$\int_{\mathbb{R}^{n}} e^{i\{x\cdot\xi\}}(2^{j\varrho}\partial_{\xi})^{\beta}(2^{j\varrho}\partial_{\eta})^{\gamma}\left[\xi^{\alpha}a_{j,\ell}(x,\xi,\eta)\right] e^{-i[c_Q,y]_t\cdot\xi}\psi_0(\xi/2^{j+1})d\xi:=T_{A^{\eta}_{j,\ell}}H(x),$$
where $T_{a}$ denotes the pseudo-differential operator defined by $(\ref{D1}).$ Setting
\begin{center}
$\bar{A}^{\eta}_{j,\ell}(x,\xi)=A^{\eta}_{j,\ell}(2^{-j\varrho}x,2^{j\varrho}\xi)$
\end{center}
we have
\begin{eqnarray}\label{E9}
\|T_{A^{\eta}_{j,\ell}}\|_{\mathcal{L}(L^{2},L^{2})}
&=&
\|T_{\bar{A}^{\eta}_{j,\ell}}\|_{\mathcal{L}(L^{2},L^{2})}.
\end{eqnarray}
Then, by the fact that
\begin{center}
$|\partial^{\alpha'}_{x}\partial^{\beta'}_{\xi}\bar{A}^{\eta}_{j,\ell}(x,\xi)|\lesssim 2^{-j\frac{n(1-\varrho)}{p}-j\frac{n}{2}\max\{\delta-\varrho,0\}+j|\alpha|} 2^{|\alpha'|j(\delta-\varrho)}$, \quad $\forall \alpha',\beta'\in \mathbb{N}^{n}$
\end{center}
and Lemma \ref{La23}, it follows that
\begin{eqnarray*}
\|T_{A^{\eta}_{j,\ell}}H\|^{2}_{L^{2}}&\lesssim& 2^{j\varrho n} 2^{-2j\frac{n(1-\varrho)}{p}-jn(\delta-\varrho)+2j|\alpha|} 2^{2\frac{n}{2N_{0}}N_{0}j(\delta-\varrho)}
\|H\|^{2}_{L^{2}}\\[2mm]
&\lesssim& 2^{-2j\frac{n(1-\varrho)}{p}+2j|\alpha|+jn},
\end{eqnarray*}
where $N_{0}=[n/2]+1.$ Combining this with (\ref{E28}), one readily obtains (\ref{E26}).

\begin{lemma}\label{La23}
Let $a(x,\xi)$ be a bounded continuous function with bounded continuous derivatives $D^{\alpha}_{x}D^{\beta}_{\xi}$ of order $|\alpha|\leq[n/2]+1:=N_{0},|\beta|\leq n+1,$ and assume that $a(x,\xi)=0$ if $\xi$ is large enough. Then the pseudo-differential operators $T_{a}$ defined by (\ref{D1}) satisfy
\begin{eqnarray}
\|T_{a}\|_{\mathcal{L}(L^{2})}\leq Cp(a)^{1-\frac{n}{2N_{0}}}q(a)^{\frac{n}{2N_{0}}}
\end{eqnarray}
where
$p(a)=\sum\limits_{|\gamma|\leq n+1}\|D^{\gamma}_{\xi}a\|_{L^{\infty}(\mathbb{R}^{2n})},$
$q(a)=\sum\limits_{|\alpha|\leq N_{0},|\gamma|\leq n+1}\|D^{\alpha}_{x}D^{\gamma}_{\xi}a\|_{L^{\infty}(\mathbb{R}^{2n})}.$
\end{lemma}
The proof of this lemma is parallel to that of Lemma 4 in \cite{Hounie2} which is omitted here.

\section{Proof of Theorem \ref{T6}}
Let $\varrho,\delta,r_{1},r_{2},m$ be given as Theorem \ref{T6}, and set $\frac{1}{r}=\frac{1}{\min\{p,2\}}+\frac{1}{\min\{q,2\}}.$ Decompose the operator $T_{a}$ as (\ref{E1}) and let $l<1$ denote the side length of some fixed cube. We shall prove the convergence of the sum over $j$ on each of the following three types of regions: $2^j < l^{-1}$, $l^{-1} \leq 2^j \leq l^{-c}$, and $2^j > l^{-c}$, where
$c=\left\{
  \begin{array}{ll}
   \frac{1}{\varrho}, & \hbox{$0<\varrho\leq1$} \\
    \frac{2}{r(1-\delta)}, & \hbox{$\varrho=0$}
  \end{array}
\right.$.
We will tackle them one by one, and the region $2^j < l^{-1}$ is treated firstly.

\begin{lemma}\label{La11}
Let $Q(x_{0},l)$ be a fixed cube with side length $l<1$. Suppose $0\leq\varrho\leq1,$ $0\leq\delta<1, 1\leq p,q\leq\infty $ and $\vec{r}=(\min\{p,2\},\min\{p,2\})$.
If
$a(x,\xi,\eta)\in BS^{-n(1-\varrho)(\frac{1}{\min\{p,2\}}+\frac{1}{\min\{q,2\}})}_{\varrho,\delta}$. Then for any positive integer $j$ satisfying $2^{j}<l^{-1}$
\begin{eqnarray}\label{E55}
|T_{a_{j}}(f,g)(x)-T_{a_{j}}(f,g)(w)|\lesssim 2^{j}l\mathcal{M}_{\vec{r}}(f,g)(x),\quad \forall x,w\in Q(x_{0},l).
\end{eqnarray}
\end{lemma}

\begin{proof}
Decompose $T_{a_{j}}$ as (\ref{eq:decomp-c}), and set
$$\tilde{K}_{j}^{\nu_{1},\nu_{2}}(x,y,z,w)=\int_{\mathbb{R}^{2n} } e^{ i  (x-y)\cdot\xi} e^{ i  (x-z)\cdot\eta}(x-w)\cdot \nabla_{x}a_{j}^{\nu_{1},\nu_{2}}(\bar{x},\xi,\eta) d\xi d\eta$$
and
$$\bar{K}_{j}^{\nu_{1},\nu_{2}}(\bar{x},y,z,w)
=\int_{\mathbb{R}^{2n} } e^{ i  (\bar{x}-y)\cdot\xi} e^{ i  (\bar{x}-z)\cdot\eta} (x-w)\cdot(\xi+\eta)a_{j}^{\nu_{1},\nu_{2}}(w,\xi,\eta) d\xi d\eta,$$
where $\tilde{x},\bar{x}$ are some point between $x$ and $w$.
Then one can write
\begin{eqnarray*}
&&T_{a_{j}}(f,g)(x)-T_{a_{j}}(f,g)(w)\\
&&=\sum\limits_{\nu_{1}\in \Lambda_{j,1}}\sum\limits_{\nu_{2}\in \Lambda_{j,2}} \int_{\mathbb{R}^{2n}}(\tilde{K}_{j}^{\nu_{1},\nu_{2}}(x,y,z,w)+\bar{K}_{j}^{\nu_{1},\nu_{2}}(\bar{x},y,z,w))f_{j}^{\nu_{1}}(y)g_{j}^{\nu_{2}}(z)dydz,
\end{eqnarray*}
where $\widehat{f_{j}^{\nu_{1}}}(\xi)=\phi(2^{-j\varrho}\xi-\nu_{1})\hat{f}(\xi)$ and $\widehat{g_{j}^{\nu_{1}}}(\eta)=\phi(2^{-j\varrho}\eta-\nu_{2})\hat{g}(\eta).$ Thus the proof of (\ref{E55}) reduces to establishing
\begin{eqnarray}\label{E56}
|\sum\limits_{\nu_{1}\in \Lambda_{j,1}}\sum\limits_{\nu_{2}\in \Lambda_{j,2}} \int_{\mathbb{R}^{2n}}\tilde{K}_{j}^{\nu_{1},\nu_{2}}(x,y,z,w)f_{j}^{\nu_{1}}(y)g_{j}^{\nu_{2}}(z)dydz|\lesssim2^{j\delta}|x-w|\mathcal{M}_{\vec{r}}(f,g)(x)
\end{eqnarray}
and
\begin{eqnarray}\label{E57}
|\sum\limits_{\nu_{1}\in \Lambda_{j,1}}\sum\limits_{\nu_{2}\in \Lambda_{j,2}} \int_{\mathbb{R}^{2n}}\bar{K}_{j}^{\nu_{1},\nu_{2}}(x,y,z,w)f_{j}^{\nu_{1}}(y)g_{j}^{\nu_{2}}(z)dydz|\lesssim2^{j}|x-w|\mathcal{M}_{\vec{r}}(f,g)(x).
\end{eqnarray}
We shall prove only (\ref{E56}), since (\ref{E57}) can be obtained by a similar argument.

Let $m=-n(1-\varrho)(\frac{1}{\min\{p,2\}}+\frac{1}{\min\{q,2\}}).$
Integrate by parts and H\"{o}lder's inequality give that the left hand on (\ref{E56}) is bounded by
\begin{eqnarray*}
&\lesssim&2^{jm+j\delta}|x-w|\sum\limits_{\nu_{1}\in \Lambda_{j,1}}\int_{\mathbb{R}^{n}}\frac{2^{j\varrho n}|f_{j}^{\nu_{1}}(y)|}{(1+2^{j\varrho}|y-x|)^{n+1}}dy
\sum\limits_{\nu_{2}\in \Lambda_{j,2}}\int_{\mathbb{R}^{n}}\frac{2^{j\varrho n}|g_{j}^{\nu_{2}}(z)|}{(1+2^{j\varrho}|z-x|)^{n+1}}dz\\
&\lesssim&2^{jm+j\delta}|x-w||\Lambda_{j,1}|^{\frac{1}{\min\{p,2\}}}|\Lambda_{j,2}|^{\frac{1}{\min\{q,2\}}}\big(\sum\limits_{\nu_{1}\in \Lambda_{j,1}}\int_{\mathbb{R}^{n}}\frac{2^{j\varrho n}|f_{j}^{\nu_{1}}(y)|^{p'}}{(1+2^{j\varrho}|y-x|)^{n+1}}dy\big)^{\frac{1}{p'}}\\
&\times&
\big(\sum\limits_{\nu_{2}\in \Lambda_{j,2}}\int_{\mathbb{R}^{n}}\frac{2^{j\varrho n}|g_{j}^{\nu_{2}}(z)|^{q'}}{(1+2^{j\varrho}|z-x|)^{n+1}}dz\big)^{\frac{1}{q'}},
\end{eqnarray*}
where $\frac{1}{p'}=1-\frac{1}{\min\{p,2\}}$ and $\frac{1}{q'}=1-\frac{1}{\min\{q,2\}}$. Set $\varphi_{N,2^{j\varrho}}(x)=\frac{2^{j\varrho n}}{(1+2^{j\varrho}|x|)^{N}}$. Then Lemma \ref{La0} and Remark \ref{r1} give that
\begin{eqnarray*}
&&\big(\sum\limits_{\nu_{1}\in \Lambda_{j,1}}\int_{\mathbb{R}^{n}}\frac{2^{j\varrho n}|f_{j}^{\nu_{1}}(y)|^{p'}}{(1+2^{j\varrho}|y-x|)^{n+1}}dy\big)^{\frac{1}{p'}}
\big(\sum\limits_{\nu_{2}\in \Lambda_{j,2}}\int_{\mathbb{R}^{n}}\frac{2^{j\varrho n}|g_{j}^{\nu_{2}}(z)|^{q'}}{(1+2^{j\varrho}|z-x|)^{n+1}}dz\big)^{\frac{1}{q'}}\\
&\lesssim&\bigg(\varphi_{n+1,2^{j\varrho}}\ast\big(\varphi_{N_{1},2^{j\varrho}}\ast|f|^{\min\{p,2\}}(x)\big)^{\frac{p'}{\min\{p,2\}}}\bigg)^{\frac{1}{p'}}
\bigg(\varphi_{n+1,2^{j\varrho}}\ast\big(\varphi_{N_{2},2^{j\varrho}}\ast|g|^{\min\{q,2\}}(x)\big)^{\frac{q'}{\min\{q,2\}}}\bigg)^{\frac{1}{q'}}\\
&\lesssim&\big(\varphi_{N_{1},2^{j\varrho}}\ast|f|^{\min\{p,2\}}(x)\big)^{\frac{1}{\min\{p,2\}}}
\big(\varphi_{N_{2},2^{j\varrho}}\ast|g|^{\min\{q,2\}}(x)\big)^{\frac{1}{\min\{q,2\}}}
\lesssim\mathcal{M}_{\vec{r}}(f,g)(x).
\end{eqnarray*}
Thus, the proof of (\ref{E56}) can be finished by the fact that
$|\Lambda_{j,1}|\lesssim 2^{j(1-\varrho)n},|\Lambda_{j,2}|\lesssim 2^{j(1-\varrho)n}$ and $m=-n(1-\varrho)(\frac{1}{\min\{p,2\}}+\frac{1}{\min\{q,2\}}).$

\end{proof}

For the region $l^{-1} \leq 2^j \leq l^{-c}$, the proof methods differ depending on whether $\delta \leq \varrho$ or $\delta > \varrho$. In the latter case, one needs to exploit the smoothness of the symbol $a(x,\xi,\eta)$ with respect to the variable $x$.
\begin{proposition}\label{P1}
Suppose $0\leq\varrho\leq1,$ $0\leq\delta<1, 2< r<\infty$ and $a\in BS^{-n(1-\frac{1}{r})-n\{\delta-\varrho,0\}}_{\varrho,\delta}$, then
\begin{eqnarray}
\|T_{a}(f,g)\|_{L^{r}{\mathbb{R}^{n}}}
&\lesssim&\|f\|_{L^{2}{\mathbb{R}^{n}}}\|g\|_{L^{2}{\mathbb{R}^{n}}}.
\end{eqnarray}
\end{proposition}
\begin{proof}
Sobolev embedding theorem gives that
\begin{eqnarray*}
\|T_{a}(f,g)\|_{L^{r}{\mathbb{R}^{n}}}
&\lesssim&\|T_{a}(f,g)\|_{L^{2}_{\frac{n}{2}-\frac{n}{\tilde{r}}}(\mathbb{R}^{n})}.
\end{eqnarray*}
Define the linear symbol $A(X,\Xi)$ as
\begin{center}
$A(X,\Xi)=a(x_{1},\xi,\eta),$
$X=(x_{1},x_{2}),\Xi=(\xi,\eta).$
\end{center}
Clearly $A(X,\Xi) \in S^{-n(1-\frac{1}{\tilde{r}})-n\{\delta-\varrho,0\}}_{\varrho,\delta}(\mathbb{R}^{2n})$, and
then generalized trace theorem for Sobolev spaces (\cite{ParkTomita2024}, Lemma 5.1) implies that
\begin{eqnarray*}
\|T_{a}(f,g)\|_{L^{2}_{\frac{n}{2}-\frac{n}{\tilde{r}}}(\mathbb{R}^{n})}
\lesssim\|T_{A(X,\Xi)}(f\otimes g)\|_{L^{2}_{n(1-\frac{1}{\tilde{r}})}(\mathbb{R}^{2n})}
\lesssim
\|f\otimes g\|_{L^{2}(\mathbb{R}^{2n})}
=
\|f\|_{L^{2}(\mathbb{R}^{n})}
\|g\|_{L^{2}(\mathbb{R}^{n})}.
\end{eqnarray*}
\end{proof}

\begin{lemma}\label{L1}
Let $Q(x_{0},l)$ be a fixed cube with side length $l<1$. Suppose $0\leq\varrho\leq1,$ $0\leq\delta<1, 1< p,q<\infty$ and $\vec{r}=(\min\{p,2\},\min\{p,2\})$. If $a\in BS^{-n(1-\varrho)(\frac{1}{\min\{p,2\}}+\frac{1}{\min\{q,2\}})-n\{\delta-\varrho,0\}}_{\varrho,\delta}$, then for any positive integer $j$ satisfying
$\left\{
  \begin{array}{ll}
    l^{-1}\leq2^{j}\leq l^{-\frac{1}{\varrho}}, & \hbox{$0<\varrho\leq1$;} \\
    l^{-1}\leq2^{j}, & \hbox{$\varrho=0$.}
  \end{array}
\right.$
\begin{eqnarray}
\frac{1}{|Q|}\int_{Q(x_{0},l)}|T_{a_{j}}(f,g)(x)|dx
&\lesssim&2^{j\frac{n}{\tilde{r}}(\frac{n}{Nr}-1)}l^{\frac{n}{\tilde{r}}(\frac{n}{Nr}-1)}\mathcal{M}_{\vec{r}}(f,g)(x),
\end{eqnarray}
where $\tilde{r}>2$, $\frac{1}{r}=\frac{1}{\min\{p,2\}}+\frac{1}{\min\{q,2\}}$ and positive integer $N>\frac{n}{r}.$
\end{lemma}

\begin{proof}
Let $\tilde{r},r,N$ be given in the lemma above and denote
\begin{eqnarray}\label{E22}
\Gamma=l^{\frac{n}{\tilde{r}N}}2^{j(\frac{n}{\tilde{r}N}-\varrho)}.
\end{eqnarray}
Then the fact that $j$ satisfies $\left\{
  \begin{array}{ll}
    l^{-1}\leq2^{j}\leq l^{-\frac{1}{\varrho}}, & \hbox{$0<\varrho\leq1$} \\
    l^{-1}\leq2^{j}, & \hbox{$\varrho=0$}
  \end{array}
\right.$ gives
$$\Gamma>l.$$
Set
$$\mathbf{Q}(X_{0},2\Gamma)=\{(y,z):y\in Q(x_{0},2\Gamma),z\in Q(x_{0},2\Gamma)\},$$
Notice that
$$\mathbf{1}_{\mathbf{Q}(X_{0},2\Gamma)}(y,z)=\mathbf{1}_{Q(x_{0},2\Gamma)}(y)\mathbf{1}_{Q(x_{0},2\Gamma)}(z),$$
we have
\begin{eqnarray}\label{E16}
\frac{1}{|Q|}\int_{Q(x_{0},l)}|T_{a_{j}}(f,g)(x)|dx
&\leq&\frac{1}{|Q|}\int_{Q(x_{0},l)}|T_{a_{j}}(f\mathbf{1}_{Q(x_{0},2\Gamma)},g\mathbf{1}_{Q(x_{0},2\Gamma)})(x)|dx\nonumber\\
&+&
\sup\limits_{x\in Q(x_{0},l)}\int_{\mathbf{Q}(X_{0},2\Gamma)^{C}}|K_{j}(x,x-y,x-z)|f(y)||g(z)|dydz.\nonumber\\
\end{eqnarray}

Next, we claim that
\begin{eqnarray*}
\frac{1}{|Q|}\int_{Q(x_{0},l)}|T_{a_{j}}(f\mathbf{1}_{Q(x_{0},2\Gamma)},g\mathbf{1}_{Q(x_{0},2\Gamma)})(x)|dx
&\lesssim&2^{j\frac{n}{\tilde{r}}(\frac{n}{Nr}-1)}l^{\frac{n}{\tilde{r}}(\frac{n}{Nr}-1)}\mathcal{M}_{\vec{r}}(f,g)(x)
\end{eqnarray*}
and
\begin{eqnarray*}
\sup\limits_{x\in Q(x_{0},l)}\int_{\mathbf{Q}(X_{0},2\Gamma)^{C}}|K_{j}(x,x-y,x-z)|f(y)||g(z)|dydz
&\lesssim&2^{j\frac{n}{\tilde{r}}(\frac{n}{Nr}-1)}l^{\frac{n}{\tilde{r}}(\frac{n}{Nr}-1)}\mathcal{M}_{\vec{r}}(f,g)(x).
\end{eqnarray*}

For the first claim, we use function $\tilde{\phi}\in \mathcal{S}(\mathbb{R}^{n})$ satisfying $\bar{\phi}=1$ on $\{\eta\in \mathbb{R}^{n}:\frac{1}{4}\leq|\eta|\leq4\}$ and $\operatorname{supp}\bar{\phi}\subset\{\eta\in \mathbb{R}^{n}:\frac{1}{8}\leq|\eta|\leq8\}$, and set $\bar{\phi}_{j}(\eta)=\bar{\phi}(\eta/2^{j})$. Decompose  $a_{j}$ as
\begin{center}
$a_{j}(x,\xi,\eta)=
\left\{
\begin{array}{ll}
\sum\limits_{k_{1}\leq j(1-\varrho)+2}a_{j}(x,\xi,\eta)\bar{\phi}_{k_{1}}(2^{-j\varrho}\xi), & \hbox{$1<p<2,q=2$;} \\
\sum\limits_{k_{2}\leq j(1-\varrho)+2}a_{j}(x,\xi,\eta)\bar{\phi}_{k_{2}}(2^{-j\varrho}\eta), & \hbox{$1<q<2,p=2$;} \\
\sum\limits_{k_{1},k_{2}\leq j(1-\varrho)+2}a_{j}(x,\xi,\eta)\bar{\phi}_k(2^{-j\varrho}\xi)\bar{\phi}_{k_{2}}(2^{-j\varrho}\eta), & \hbox{$1<p<2,1<p<2$.}
\end{array}
\right.
$
\end{center}
Set
\begin{center}
$\hat{f}_{1}(\xi)=
\left\{
\begin{array}{ll}
2^{-(k_{1}+j\varrho)(\frac{n}{p}-\frac{n}{2})}\widehat{f\mathbf{1}_{Q(x_{0},2\Gamma)}}(\xi)
\mathbf{1}_{\operatorname{supp}\bar{\phi}_{k_{1}}(2^{-j\varrho}\cdot)}(\xi), & \hbox{$1<p<2$;} \\
\widehat{f\mathbf{1}_{Q(x_{0},2\Gamma)}}(\xi), & \hbox{$p=2$.}
\end{array}
\right.
$
\end{center}
and
\begin{center}
$\hat{g}_{1}(\eta)=
\left\{
\begin{array}{ll}
2^{-(k_{2}+j\varrho)(\frac{n}{q}-\frac{n}{2})}\widehat{g\mathbf{1}_{Q(x_{0},2\Gamma)}}(\eta)
\mathbf{1}_{\operatorname{supp}\bar{\phi}_{k_{2}}(2^{-j\varrho}\cdot)}(\eta), & \hbox{$1<q<2$;} \\
\widehat{g\mathbf{1}_{Q(x_{0},2\Gamma)}}(\eta), & \hbox{$q=2$.}
\end{array}
\right.
$
\end{center}
Then one can write
$$T_{a_{j}}(f\mathbf{1}_{Q(x_{0},2\Gamma)},g\mathbf{1}_{Q(x_{0},2\Gamma)})(x)
=T_{a_{j,k_{1},k_{2}}}(f_{1},g_{1})(x),$$
and for $1<p,q\leq2$
\begin{center}
$\|f_{1}\|_{L^{2}}\lesssim\|f\mathbf{1}_{Q(x_{0},2\Gamma)}\|_{L^{p}}$
and
$\|g_{1}\|_{L^{2}}\lesssim\|g\mathbf{1}_{Q(x_{0},2\Gamma)}\|_{L^{q}}.$
\end{center}

We only consider the case when $1<p<2,1<q<2$, and the remain case: $1<p<2,q=2$ and $1<q<2,p=2$ can be treated by the same argument.

Set
\begin{center}
$a_{j,k_{1},k_{2}}(x,\xi,\eta)=a_{j}(x,\xi,\eta)\bar{\phi}_k(2^{-j\varrho}\xi)\bar{\phi}_{k_{2}}(2^{-j\varrho}\eta)
2^{(k_{1}+j\varrho)(\frac{n}{p}-\frac{n}{2})}2^{(k_{2}+j\varrho)(\frac{n}{q}-\frac{n}{2})}$
\end{center}
and define the linear symbol $A_{j,k}(X,\Xi)$ as
\begin{center}
$A_{j,k_{1},k_{2}}(X,\Xi)=a_{j,k_{1},k_{2}}(x,\xi,\eta),$
$X=(x_{1},x_{2}),\Xi=(\xi,\eta).$
\end{center}
Then $A_{j,k_{1},k_{2}}(X,\Xi) \in S^{-n(1-\frac{1}{\tilde{r}})-n\{\delta-\varrho,0\}}_{\varrho,\delta}(\mathbb{R}^{2n})$ with bounds
$$\lesssim 2^{jm+(k_{1}+j\varrho)(\frac{1}{p}-\frac{1}{2})n+(k_{2}+j\varrho)(\frac{1}{q}-\frac{1}{2})n+jn(1-\frac{1}{\tilde{r}})}$$
 for all $j,k$ since
\begin{center}
$|\xi|+|\eta|\approx 2^{j}.$
\end{center}
Thus Proposition \ref{P1} gives that
\begin{eqnarray*}
&&\|T_{a_{j}}(f\mathbf{1}_{Q(x_{0},2\Gamma)},g\mathbf{1}_{Q(x_{0},2\Gamma)})\|_{L^{\tilde{r}}(\mathbb{R}^{n})}
=\|T_{a_{j,k_{1},k_{2}}}(f_{1},g_{1})\|_{L^{\tilde{r}}(\mathbb{R}^{n})}\\
&\lesssim& 2^{jm+(k_{1}+j\varrho)(\frac{1}{p}-\frac{1}{2})n+(k_{2}+j\varrho)(\frac{1}{q}-\frac{1}{2})n+jn(1-\frac{1}{\tilde{r}})}
\|f\mathbf{1}_{Q(x_{0},2\Gamma)}\|_{L^{p}}
\|g\mathbf{1}_{Q(x_{0},2\Gamma)}\|_{L^{q}}.
\end{eqnarray*}
Moreover,
\begin{eqnarray*}
&&\sum\limits_{k_{1},k_{2}\leq j(1-\varrho)+2}\frac{1}{|Q|}\int_{Q(x_{0},l)}|T_{a_{j,2,k}}(f_{1},g_{1})(x)|dx\\
&\lesssim&\sum\limits_{k_{1},k_{2}\leq j(1-\varrho)+2}l^{-\frac{n}{\tilde{r}}}\|T_{a_{j,2,k}}(f_{1},g_{1})\|_{L^{\tilde{r}}(\mathbb{R}^{n})}\\
&\lesssim&\sum\limits_{k_{1},k_{2}\leq j(1-\varrho)+2}l^{-\frac{n}{\tilde{r}}}
2^{jm+(k_{1}+j\varrho)(\frac{1}{p}-\frac{1}{2})n+(k_{1}+j\varrho)(\frac{1}{q}-\frac{1}{2})n+jn(1-\frac{1}{\tilde{r}})}
\|f\mathbf{1}_{Q(x_{0},2\Gamma)}\|_{L^{p}}
\|g\mathbf{1}_{Q(x_{0},2\Gamma)}\|_{L^{q}}\\
&\lesssim&2^{j\frac{n}{\tilde{r}}(\frac{n}{Nr}-1)}l^{\frac{n}{\tilde{r}}(\frac{n}{Nr}-1)}\mathcal{M}_{\vec{r}}(f,g)(x).
\end{eqnarray*}

For the second claim, Integrating by parts shows that it can be bounded by
$$\sup\limits_{x\in Q(x_{0},l)}\int_{\mathbf{Q}(X_{0},2\Gamma)^{C}}\frac{1}{|(y,z)-(x,x)|^{N}}|K_{j,N}(x,x-y,x-z)|f(y)||g(z)|dydz,$$
where
$$K_{j,N}(x,x-y,x-z)=\sup\limits_{|\alpha|=N}\int_{\mathbb{R}^{2n} } e^{ i  ((x,x)-(y,z))\cdot(\xi,\eta)} \partial^{\alpha}_{(\xi,\eta)}a_{j}(x,\xi,\eta) d\xi d\eta.$$
Notice that $|(y,z)-(x,x)|\approx|(y,z)-(x_{0},x_{0})|$ for $\forall x\in Q(x_{0},l)$ and $\forall (y,z)\in Q^{C}\big((x_{0},x_{0}),2\Gamma\big)$ since $\Gamma>l$. So a direct computations gives that
\begin{eqnarray}\label{E21}
&&\sup\limits_{x\in Q(x_{0},l)}\sum\limits_{k=1}^{\infty}\frac{1}{(2^{k}\Gamma)^{N}}\int_{|(y,z)-(x_{0},x_{0})|\approx 2^{k}\Gamma}|K_{j,N}(x,x-y,x-z)|f(y)||g(z)|dydz\nonumber\\
&\leq&\sup\limits_{x\in Q(x_{0},l)}\sum\limits_{k=1}^{\infty}\frac{1}{(2^{k}\Gamma)^{N}}
\int_{|z-x_{0}|\leq 2^{k}\Gamma}\int_{|y-x_{0}|\leq 2^{k}\Gamma}|K_{j,N}(x,x-y,x-z)|f(y)||g(z)|dydz\nonumber\\
&\leq&\sup\limits_{x\in Q(x_{0},l)}\sum\limits_{k=1}^{\infty}\frac{1}{(2^{k}\Gamma)^{N}}
\big(\int_{\mathbb{R}^{n}}(\int_{\mathbb{R}^{n}}|K_{j,N}(x,x-y,x-z)|^{p'}dy)^{\frac{q'}{p'}}dz\big)^{\frac{1}{q'}}\nonumber\\
&&\times\big(\int_{|y-x_{0}|\leq 2^{k}\Gamma}|f(y)|^{p}dy\big)^{\frac{1}{p}}
\big(\int_{|z-x_{0}|\leq 2^{k}\Gamma}|g(z)|^{q}dz\big)^{\frac{1}{q}}\nonumber\\
&\lesssim&2^{j\frac{n}{\tilde{r}}(\frac{n}{Nr}-1)}l^{\frac{n}{\tilde{r}}(\frac{n}{Nr}-1)}\mathcal{M}_{\vec{r}}(f,g)(x).\nonumber
\end{eqnarray}
Thus, the proof in this case is finished.
\end{proof}

\begin{lemma}\label{La22}
Let $Q(x_{0},l)$ be a fixed cube with side length $l<1$. Suppose $0\leq\varrho<\delta<1,1< p,q<\infty$ and $\vec{r}=(\min\{p,2\},\min\{p,2\})$. If $a\in BS^{-n(1-\varrho)(\frac{1}{\min\{p,2\}}+\frac{1}{\min\{q,2\}})}_{\varrho,\delta}$, $\lambda$ satisfying
$\left\{
  \begin{array}{ll}
    1\leq\lambda\leq\frac{1}{\varrho}, & \hbox{$0<\varrho\leq1$} \\
    1\leq\lambda<\infty, & \hbox{$\varrho=0$}
  \end{array},
\right.$
and integer $j$ satisfying
$\left\{
  \begin{array}{ll}
    l^{-\lambda}\leq2^{j}\leq l^{-\frac{1}{\varrho}}, & \hbox{$0<\varrho\leq1$} \\
    l^{-\lambda}\leq2^{j}, & \hbox{$\varrho=0$}
  \end{array}.
\right.$
Then
\begin{eqnarray*}
\frac{1}{|Q|}\int_{Q(x_{0},l)}|T_{a_{j}}(f,g)(x)|dx
&\lesssim&\big(2^{j\delta}l^{\lambda}+2^{j\frac{n}{\tilde{r}}(\frac{n}{Nr}-1)}l^{\frac{\lambda n}{\tilde{r}}(\frac{n}{Nr}-1)}\big)\mathcal{M}_{\vec{r}}(f,g)(x).
\end{eqnarray*}
\end{lemma}
\begin{proof}
Notice that $l^{\lambda}<l$ when $\lambda>1$. Take integer $L$ such that it is the first number no less than  $l^{1-\lambda}$, that is $L-1<l^{1-\lambda}\leq L$. Then there are $L^{n}$ cubes with the same side length $l^{\lambda}$ covering $Q(x_{0},l)$. Moreover, we have
$$Q(x_{0},l)\subset\cup_{i=1}^{L^{n}}Q(x_{i},l^{\lambda})\subset Q(x_{0},2l).$$
Clearly, $L^{n}\leq2^{n}l^{n(1-\lambda)}$. Denote
\begin{eqnarray}
T_{a_{j,i}}(f,g)(x)
&=&\int_{\mathbb{R}^{2n} }e^{ i x\cdot(\xi+\eta)}a_{j}(x_{i},\xi,\eta) \hat{f}(\xi) \hat{g}(\eta)d\xi d\eta.
\end{eqnarray}
and write
\begin{eqnarray}\label{E18}
&&\frac{1}{|Q|}\int_{Q(x_{0},l)}|T_{a_{j}}(f,g)(x)|dx\nonumber\\
&\leq&\frac{1}{|Q|}\sum\limits_{i=1}^{L^{n}}\bigg(\int_{Q(x_{i},l^{\lambda})}|T_{a_{j}}(f,g)(x)-T_{a_{j,i}}(f,g)(x)|dx+\int_{Q(x_{i},l^{\lambda})}|T_{a_{j,i}}(f,g)(x)|dx\bigg).
\nonumber\\
\end{eqnarray}
By a similar  argument as \eqref{E55}, we have
\begin{eqnarray}\label{E19}
|T_{a_{j}}(f,g)(x)-T_{a_{j,i}}(f,g)(x)|\lesssim|x-x_{i}|2^{j\delta}\mathcal{M}_{\vec{r}}(f,g)(x),
\end{eqnarray}
From the fact that $a_{j}(x_{i},\xi,\eta) \in BS^{-n(1-\varrho)(\frac{1}{\min\{p,2\}}+\frac{1}{\min\{q,2\}})}_{\varrho,0}$, it follows that by Lemma \ref{L1}
\begin{eqnarray}\label{E20}
\frac{1}{|Q(x_{i},l^{\lambda})|}\int_{Q(x_{i},l^{\lambda})}|T_{a_{j,i}}(f,g)(x)|dx
&\lesssim&2^{j\frac{n}{\tilde{r}}(\frac{n}{Nr}-1)}l^{\lambda\frac{n}{\tilde{r}}(\frac{n}{Nr}-1)}\mathcal{M}_{\vec{r}}(f,g)(x).
\end{eqnarray}
Since $L^{n}\leq2^{n}l^{n(1-\lambda)}$, we can get the desired estimate by substituting both (\ref{E19}) and (\ref{E20}) into (\ref{E18}).

When $\lambda=1$, we define
\begin{eqnarray}
T_{a_{j,0}}(f,g)(x)
&=&\int_{\mathbb{R}^{2n} }e^{ i x\cdot(\xi+\eta)}a_{j}(x_{0},\xi,\eta) \hat{f}(\xi) \hat{g}(\eta)d\xi d\eta.
\end{eqnarray}
Then the desired estimate can be proved by the same argument as above with $T_{a_{j,0}}$ replaced by $T_{a_{j,0}}$. So we complete the proof.
\end{proof}

Last, we deal with the region $2^j > l^{-c}$.
\begin{lemma}\label{La12}
Let $Q(x_{0},l)$ be a fixed cube with side length $l<1$. Suppose  $0\leq\varrho\leq1$, $0\leq\delta<1$, $1< p,q<\infty$ and $\vec{r}=(\min\{p,2\},\min\{p,2\})$. If $a\in BS^{-n(1-\varrho)(\frac{1}{\min\{p,2\}}+\frac{1}{\min\{q,2\}})}_{\varrho,\delta}$, then for any positive integer $j$ satisfying
$\left\{
  \begin{array}{ll}
   l^{-\frac{1}{\varrho}} \leq2^{j}, & \hbox{$0<\varrho\leq1$;} \\
    l^{-\frac{2}{r(1-\delta)}}\leq2^{j}, & \hbox{$\varrho=0$,}
  \end{array}
\right.$

\begin{eqnarray*}
&&\frac{1}{|Q|}\int_{Q(x_{0},l)}|T_{a_{j}}(f,g)(x)|dx\\
&\lesssim&\mathcal{M}_{\vec{r}}(f,g)(x)
\left\{
  \begin{array}{ll}
   2^{-j(\frac{n}{2}(1-\varrho)-\frac{n}{2}\max\{\delta-\varrho,0\})}+2^{-j\varrho(\frac{n}{r}-N}l^{\frac{n}{r}-N}, & \hbox{$0<\varrho\leq1$} \\
    2^{-j\frac{n}{2}(1-\delta)(1-\frac{n}{rN})}l^{-\frac{n}{r}(1-\frac{n}{pN})}, & \hbox{$\varrho=0$}
  \end{array},
\right.
\end{eqnarray*}
where $\frac{1}{r}=\frac{1}{\min\{p,2\}}+\frac{1}{\min\{q,2\}}$ and positive integer $N>\frac{n}{r}.$
\end{lemma}
\begin{proof}
Take
$\Gamma=\left\{
  \begin{array}{ll}
   l, & \hbox{$0<\varrho\leq1$} \\
    2^{j\frac{n}{rN}(1-\delta)}l^{\frac{n}{rN}}, & \hbox{$\varrho=0$}
  \end{array}
\right.$ in (\ref{E22}), and then it suffices to show that
\begin{eqnarray}\label{E23}
&&\frac{1}{|Q|}\int_{Q(x_{0},l)}|T_{a_{j}}(f\mathbf{1}_{Q(x_{0},2\Gamma)},g\mathbf{1}_{Q(x_{0},2\Gamma)})(x)|dx\nonumber\\
&&\lesssim\mathcal{M}_{\vec{r}}(f,g)(x)
\left\{
  \begin{array}{ll}
   2^{-j(\frac{n}{2}(1-\varrho)-\frac{n}{2}\max\{\delta-\varrho,0\})}, & \hbox{$0<\varrho\leq1$} \\
    2^{-j\frac{n}{2}(1-\delta)(1-\frac{n}{rN})}l^{-\frac{n}{r}(1-\frac{n}{pN})}, & \hbox{$\varrho=0$}
  \end{array}
\right.
\end{eqnarray}
and
\begin{eqnarray}\label{E24}
&&\sup\limits_{x\in Q(x_{0},l)}\int_{\mathbf{Q}(X_{0},2\Gamma)^{C}}|K_{j}(x,x-y,x-z)|f(y)||g(z)|dydz\nonumber\\
&&\lesssim\mathcal{M}_{\vec{r}}(f,g)(x)
\left\{
  \begin{array}{ll}
   2^{-j\varrho(\frac{n}{r}-N}l^{\frac{n}{r}-N}, & \hbox{$0<\varrho\leq1$} \\
    2^{-j\frac{n}{2}(1-\delta)(1-\frac{n}{rN})}l^{-\frac{n}{r}(1-\frac{n}{pN})}, & \hbox{$\varrho=0$}
  \end{array}.
\right.
\end{eqnarray}

The inequality (\ref{E24}) is the consequence result of (\ref{E21}). For (\ref{E23}), by Theorem \ref{thm:main},
we have
\begin{eqnarray*}
&&\frac{1}{|Q|}\int_{Q(x_{0},l)}|T_{a_{j}}(f\mathbf{1}_{Q(x_{0},2\Gamma)},g\mathbf{1}_{Q(x_{0},2\Gamma)})(x)|dx\\
&\lesssim&l^{-\frac{n}{r}}\|T_{a_{j}}(f\mathbf{1}_{Q(x_{0},2\Gamma)},g\mathbf{1}_{Q(x_{0},2\Gamma)})\|_{L^{r}(\mathbb{R}^{n})}\\
&\lesssim&l^{-\frac{n}{r}}2^{-j(\frac{n}{2}(1-\varrho)-\frac{n}{2}\max\{\delta-\varrho,0\})}
\|f\mathbf{1}_{Q(x_{0},2\Gamma)}\|_{L^{p}(\mathbb{R}^{n})}
\|g\mathbf{1}_{Q(x_{0},2\Gamma)}\|_{L^{q}(\mathbb{R}^{n})}\\
&\lesssim&l^{-\frac{n}{r}}\Gamma^{\frac{n}{r}}2^{-j(\frac{n}{2}(1-\varrho)-\frac{n}{2}\max\{\delta-\varrho,0\})}\mathcal{M}_{\vec{r}}(f,g)(x).
\end{eqnarray*}
Here, we use the fact that $a_{j}(x,\xi,\eta)\in BS^{m_\varrho(p,q)-n\frac{\max\{\delta-\varrho,0\}}{\max\{r,2\}}}_{\varrho,\delta}(\mathbb{R}^{n})$ with bounds
$$\lesssim 2^{-j(\frac{n}{2}(1-\varrho)-\frac{n}{2}\max\{\delta-\varrho,0\})}.$$
Thus the proof is finished.
\end{proof}

{\bf The proof of Theorem \ref{T6}}
Without loss of generality, we assume that the symbol $a(x,\xi,\eta)$ vanishes for $|\xi|+|\eta|\leq 1$. Let $Q=Q(x_{0},l)$ denote the cube centered at $x_{0}$ with the side length $l.$ For any fixed cube $Q$, we are going to prove that
\begin{eqnarray}\label{E7}
\frac{1}{|Q|}\int_{Q}|T_a(f,g)(x)-C_{Q}|dx\lesssim \mathcal{M}_{\vec{r}}(f,g)(x),
\end{eqnarray}
where $C_{Q}=\frac{1}{|Q|}\int_{Q}T_a(f,g)(w)dw$. As usual the proof will be divided into two cases: $0<l<1$ and $l\geq1.$ The case when $l\geq1$ can be treated by a standard method. We put our our eyes on the case when $0<l<1$.

We decompose the operator $T_{a}$ as (\ref{E1}), then the left hand of  (\ref{E7}) is bounded by
\begin{eqnarray}\label{E8}
\sum\limits_{j=1}^{\infty}\frac{1}{|Q|^{2}}\int_{Q}\int_{Q}|T_{a_{j}}(f,g)(x)-T_{a_{j}}(f,g)(w)|dwdx.
\end{eqnarray}
Break the sum as
$$\sum\limits_{1\leq2^{j}\leq l^{-1}}+\sum\limits_{l^{-1}<2^{j}}.$$
The first summation can be treated by Lemma \ref{La11}.

Applying the following estimate:
\begin{equation*}
\begin{array}{c}
  \displaystyle \int_{Q}\int_{Q}|T_{a_{j}}(f,g)(x)-T_{a_{j}}(f,g)(w)|dwdx\leq  \frac{2}{|Q|^{2}}\int_{Q}|T_{a_{j}}(f,g)(x)|dx,
\end{array}
\end{equation*}
we consider two cases to estimate the second summation $\sum\limits_{l^{-1}<2^{j}}$.

{\bf Case 1:}~$0\leq\varrho\leq1,0\leq\delta<1$ and $\delta\leq\varrho$.

Clearly, This summation is convergent in the case when $\varrho=\delta=0$ by Lemma \ref{L1}, and it is also convergent in the case when $0<\delta\leq\varrho\leq1,\delta\neq1$ by Lemma \ref{L1} and Lemma \ref{La12}.

{\bf Case 2:}~$0\leq\varrho\leq1,0\leq\delta<1$ and $\delta>\varrho$.

Set $\frac{1}{r}=\frac{1}{\min\{p,2\}}+\frac{1}{\min\{q,2\}}$. Take positive integer $\gamma$ such that
\begin{center}
$\left\{
  \begin{array}{ll}
   \delta^{\gamma}\approx\varrho, & \hbox{$0<\varrho\leq1$} \\
   \delta^{\gamma}\approx \frac{r(1-\delta)}{2}, & \hbox{$\varrho=0$}
  \end{array}
\right.$
\end{center}
and write the summation $\sum\limits_{l^{-1}<2^{j}}$ as
\begin{eqnarray*}
&&\big(\sum\limits_{l^{-1}<2^{j}\leq l^{-\frac{1}{\delta}}}
+\sum\limits_{l^{-\frac{1}{\delta}}<2^{j}\leq l^{-\frac{1}{\delta^{2}}}}+...
+\sum\limits_{l^{-\frac{1}{\delta^{k-1}}}<2^{j}\leq l^{-\frac{1}{\delta^{k}}}}
+...
+\sum\limits_{l^{-\frac{1}{\delta^{\gamma-1}}}<2^{j}\leq l^{-\frac{1}{\delta^{\gamma}}}}\big)
+\sum\limits_{l^{-\frac{1}{\delta^{\gamma}}}<2^{j}}.
\end{eqnarray*}
Lemma \ref{La12} implies that the last summation is convergent.
Take $\lambda=\frac{1}{\delta^{k}}$ for $k=1,2...,\gamma-1,\gamma$ in Lemma \ref{La22} respectively, it is easy to check that each summation within the parentheses is convergent. So the proof is completed.

\section*{Appendix A}
In this section, we present the proof outlines for Theorems \ref{T2} and \ref{T4}. These outlines are derived from the proofs of Theorems 1.2 and 1.3 in \cite{MiyachiTomita2018}. In order to show respect for the author, we have used the notations employed in \cite{MiyachiTomita2018} as much as possible.

\subsection*{Proof of Theorem \ref{T2}}

Let $a_{Q}$ be an $H^p$-atom satisfying \eqref{eq:atom} with $L > n/p - n$. Denote by $c_Q$ the center of $Q$, by $\ell(Q)$ its side length, and by $Q^*$ the concentric cube expanded by a factor of $2\sqrt{n}$. The proof reduces to establishing
\begin{align}
&|T_a(a_{Q},g)(x)|\mathbf{1}_{(Q^*)^c}(x) \lesssim u(x)v(x), \quad \|u\|_{L^p} \lesssim 1, \quad \|v\|_{L^2} \lesssim \|g\|_{L^2}, \label{eq:3.2} \\
&|T_a(a_{Q},g)(x)|\mathbf{1}_{Q^*}(x) \lesssim u'(x)v'(x), \quad \|u'\|_{L^p} \lesssim 1, \quad \|v'\|_{L^2} \lesssim \|g\|_{L^2}, \label{eq:3.3}
\end{align}
where $u, u'$ depend solely on $a_{Q}$ and $v, v'$ depend only on $g$.

To obtain \eqref{eq:3.2} and \eqref{eq:3.3}, decompose $T_a(a,g)$ as
\begin{equation}\label{eq:3.5}
T_a(a,g)(x) = \sum_{j=0}^\infty \sum_{\ell=0}^\infty T_{a_{j,\ell}}(a,g)(x) = \sum_{\substack{j,\ell \geq 0 \\ \ell \leq j(1-\rho)+2}} T_{a_{j,\ell}}(a,g_{j,\ell})(x)
\end{equation}
with
\[
a_{j,\ell}(x,\xi,\eta) = a(x,\xi,\eta) \Psi_j(\xi,\eta) \psi_\ell(\eta/2^{j\rho})
\]
and
\[
g_{j,\ell}(x) = \widetilde{\psi}_\ell(D/2^{[j\rho]})g(x),
\]
where $\Psi_j$ and $\psi_\ell$ are as in \eqref{E0} with $d=2n$ and $d=n$ respectively. The functions $\widetilde{\psi}_\ell$ are defined by $\widetilde{\psi}_\ell(\eta) = \widetilde{\psi}(\eta/2^\ell)$ for $\ell \geq 1$, with a Schwartz function $\widetilde{\psi}$ satisfying $\widetilde{\psi} = 1$ on $\{\eta \in \mathbb{R}^n : 1/4 \leq |\eta| \leq 4\}$, $\operatorname{supp} \widetilde{\psi} \subset \{\eta \in \mathbb{R}^n : 1/8 \leq |\eta| \leq 8\}$, while $\widetilde{\psi}_0 = 1$ on $\{\eta \in \mathbb{R}^n : |\eta| \leq 4\}$ with $\operatorname{supp} \widetilde{\psi}_0 \subset \{\eta \in \mathbb{R}^n : |\eta| \leq 8\}$.

The partial inverse Fourier transform of $a_{j,\ell}(x,\xi,\eta)$ with respect to $(\xi,\eta)$ is
\[
K_{j,\ell}(x,y,z) = \frac{1}{(2\pi)^{2n}} \int_{(\mathbb{R}^n)^2} e^{i(y\cdot\xi+z\cdot\eta)} a_{j,\ell}(x,\xi,\eta)\, d\xi d\eta, \quad x,y,z \in \mathbb{R}^n,
\]
yielding the representation
\[
T_{a_{j,\ell}}(a,g_{j,\ell})(x) = \int_{(\mathbb{R}^n)^2} K_{j,\ell}(x,x-y,x-z) a(y) g_{j,\ell}(z)\, dy dz.
\]

\subsubsection*{Proof of \eqref{eq:3.2}}

Assume $x \notin Q^*$. Set $K_{j,\ell}^{(\alpha,0)}(x,y,z) = \partial_y^\alpha K_{j,\ell}(x,y,z)$ and $[c_Q,y]_t = c_Q + t(y-c_Q)$ for $0<t<1$. Define
\begin{align*}
h_{j,\ell}^{(Q,L)}(x) &= 2^{-j\rho n/2} \ell(Q)^{L-n/p} \sum_{|\alpha|=L} \sum_{|\beta| \leq M} \sum_{|\gamma| \leq M'} \int_{\substack{y \in Q \\ 0<t<1}} \\
&\qquad \times \left\|(2^{j\rho}(x-[c_Q,y]_t))^\beta (2^{j\rho}z)^\gamma K_{j,\ell}^{(\alpha,0)}(x,x-[c_Q,y]_t,z)\right\|_{L^2_z}\, dy dt
\end{align*}
and
\begin{equation}\label{eq:3.10}
\widetilde{g}_{j,\ell}(x) = 2^{j\rho n/2} \left\|(1+2^{j\rho}|x-\cdot|)^{-M'} g_{j,\ell}(\cdot)\right\|_{L^2},
\end{equation}
where $M > n/p - n/2$ and $M' > n/2$.

The size condition on $a_{Q}$ implies
\begin{equation}\label{eq:3.9}
|T_{a_{j,\ell}}(a,g_{j,\ell})(x)| \lesssim (1+2^{j\rho}|x-c_Q|)^{-M} h_{j,\ell}^{(Q,L)}(x) \widetilde{g}_{j,\ell}(x).
\end{equation}

Dropping the moment condition on $a$, one similarly obtains
\begin{equation}\label{eq:3.11}
|T_{a_{j,\ell}}(a,g_{j,\ell})(x)| \lesssim (1+2^{j\rho}|x-c_Q|)^{-M} h_{j,\ell}^{(Q,0)}(x) \widetilde{g}_{j,\ell}(x)
\end{equation}
with
\[
h_{j,\ell}^{(Q,0)}(x) = 2^{-j\rho n/2} \ell(Q)^{-n/p} \sum_{|\beta| \leq M} \sum_{|\gamma| \leq M'} \int_{y \in Q} \left\|(2^{j\rho}(x-y))^\beta (2^{j\rho}z)^\gamma K_{j,\ell}(x,x-y,z)\right\|_{L^2_z}\, dy.
\]

Set
\[
u_{j,\ell}(x) = (1+2^{j\rho}|x-c_Q|)^{-M} \min\left\{h_{j,\ell}^{(Q,L)}(x), h_{j,\ell}^{(Q,0)}(x)\right\}
\]
\[
u(x) = \left(\sum_{\ell \leq j(1-\rho)+2} 2^{-(\ell-j(1-\rho))2\varepsilon} u_{j,\ell}(x)^2\right)^{1/2},
\]
and
\begin{equation}\label{eq:3.15}
v(x) = \left(\sum_{\ell \leq j(1-\rho)+2} 2^{(\ell-j(1-\rho))2\varepsilon} \widetilde{g}_{j,\ell}(x)^2\right)^{1/2},
\end{equation}
where $0 < \varepsilon < n/2$. Schwarz's inequality yields
\[
|T_a(a,g)(x)| \lesssim \sum_{\ell \leq j(1-\rho)+2} u_{j,\ell}(x) \widetilde{g}_{j,\ell}(x) \leq u(x)v(x) \quad \text{for all } x \notin Q^*.
\]
Note that $u$ depends only on $a$ and $v$ depends only on $g$.

We first verify that
\[
\|v\|_{L^2} \lesssim \|g\|_{L^2}.
\]
Indeed,
\begin{align*}
\|v\|_{L^2}^2
&\approx \sum_{j=0}^\infty \sum_{\ell=0}^{[j(1-\rho)]+2} 2^{(\ell-j(1-\rho))2\varepsilon} \|g_{j,\ell}\|_{L^2}^2\\
&\lesssim \sum_{j=0}^\infty 2^{-j(1-\rho)2\varepsilon} \|g_{j,0}\|_{L^2}^2
+ \sum_{j=0}^\infty \sum_{\ell=1}^{[j(1-\rho)]+2} 2^{(\ell-j(1-\rho))2\varepsilon} \|g_{j,\ell}\|_{L^2}^2\\
&\lesssim \sum_{j=0}^\infty 2^{-j(1-\rho)2\varepsilon} \|g\|_{L^2}^2 + \sum_{k=1}^\infty \sum_{j=\max\{0,k-2\}}^\infty 2^{(k-j)2\varepsilon} \|\widetilde{\psi}(D/2^k)g\|_{L^2}^2\\
&\lesssim \|g\|_{L^2}^2.
\end{align*}

It remains to establish
\begin{equation}\label{E54}
\|u\|_{L^p} \lesssim 1,
\end{equation}
which will complete the proof of \eqref{eq:3.2}. For this purpose, it suffices to prove
\begin{equation}\label{eq:3.12}
\|h_{j,\ell}^{(Q,L)}\|_{L^2} \lesssim \left(2^j \ell(Q)\right)^{L-n/p+n} 2^{\ell n/2}
\end{equation}
and
\begin{equation}\label{eq:3.14}
\|h_{j,\ell}^{(Q,0)}\|_{L^2} \lesssim \left(2^j \ell(Q)\right)^{-n/p+n} 2^{\ell n/2}.
\end{equation}
We establish \eqref{eq:3.12}, the bound \eqref{eq:3.14} follows analogously. Observe that
\begin{align*}
\|h_{j,\ell}^{(Q,L)}\|_{L^2} &\leq 2^{-j\rho n/2} \ell(Q)^{L-n/p} \sum_{|\alpha|=L} \sum_{|\beta| \leq M} \sum_{|\gamma| \leq M'} \int_{\substack{y \in Q \\ 0<t<1}} \\
&\qquad \times \left\|(2^{j\rho}(x-[c_Q,y]_t))^\beta (2^{j\rho}z)^\gamma K_{j,\ell}^{(\alpha,0)}(x,x-[c_Q,y]_t,z)\right\|_{L^2_{x,z}}\, dy dt.
\end{align*}
Inserting \eqref{E26} into this estimate yields \eqref{eq:3.12}.

\subsubsection*{Proof of \eqref{eq:3.3}}

Fix $M' > n/2$. By Schwarz's inequality and Plancherel's theorem on $\mathbb{R}^{2n}$,
\begin{align*}
|T_{a_{j,\ell}}(a,g_{j,\ell})(x)|
&\leq |Q|^{-1/p} \left\| (1 + 2^{j\rho}|x-y|)^{M'} (1 + 2^{j\rho}|x-z|)^{M'} K_{j,\ell}(x, x-y, x-z) \right\|_{L^2_{y,z}} \\
&\quad \times \left\| (1 + 2^{j\rho}|x-y|)^{-M'} (1 + 2^{j\rho}|x-z|)^{-M'} g_{j,\ell}(z) \right\|_{L^2_{y,z}}\\
&\lesssim |Q|^{-1/p} 2^{-j(1-\rho)n(1/p-1)} 2^{(\ell-j(1-\rho))n/2} \widetilde{g}_{j,\ell}(x),
\end{align*}
where $\widetilde{g}_{j,\ell}$ is given by \eqref{eq:3.10}. With $v$ as in \eqref{eq:3.15},
\[
|T_a(a,g)(x)| \lesssim |Q|^{-1/p} \mathbf{1}_{Q^*}(x) v(x).
\]
Setting $u'(x)=|Q|^{-1/p} \mathbf{1}_{Q^*}(x)$ and $v'(x)=v(x)$, we obtain $\|u'\|_{L^{p}}\lesssim1$ and $\|v'\|_{L^{2}}\lesssim\|g\|_{L^{2}}$, completing the proof.

\subsection*{Proof of Theorem \ref{T4}}

Let $g \in L^\infty$ and let $a_{Q}$ be an $H^p$-atom satisfying \eqref{eq:atom} with a cube $Q$ centered at the origin and $L > n/p - n$. The desired boundedness follows from
\begin{equation}\label{eq:4.1}
\|T_a(a_{Q},g)\|_{L^p} \lesssim \|g\|_{L^\infty}.
\end{equation}
We decompose the $p$th power of the left-hand side of \eqref{eq:4.1} as
\begin{equation}\label{eq:4.2}
\|T_a(a_{Q},g)\|_{L^p(Q^*)}^p + \|T_a(a_{Q},g)\|_{L^p((Q^*)^c)}^p.
\end{equation}

For the first term, Theorem~\ref{TM2} yields
\[
\|T_a(a_{Q},g)\|_{L^p(Q^*)} \leq |Q^*|^{1/p-1/2} \|T_a(a_{Q},g)\|_{L^2} \lesssim |Q|^{1/p-1/2} \|a\|_{L^2} \|g\|_{L^\infty} \leq \|g\|_{L^\infty},
\]
where we use the embedding
\[
BS^{-\frac{n}{p}(1-\varrho)-\frac{n}{2}\max\{\delta-\varrho,0\}}_{\rho,\delta} \subset BS^{-\frac{n}{2}(1-\varrho)-\frac{n}{2}\max\{\delta-\varrho,0\}}_{\rho,\delta}.
\]

It remains to estimate the second term in \eqref{eq:4.2}. Decompose $T_a$ as in (\ref{E1}) with $d=2n$:
\[
\sum_{j=0}^{\infty}T_{a_{j}} \quad \text{with} \quad a_{j}(x,\xi,\eta)=a(x,\xi,\eta)\Psi_j(\xi,\eta).
\]
Set
\[
K_j(x,y,z) = \frac{1}{(2\pi)^{2n}} \int_{(\mathbb{R}^n)^2} e^{i(y\cdot\xi+z\cdot\eta)} a_j(x,\xi,\eta)\, d\xi d\eta, \quad x,y,z \in \mathbb{R}^n,
\]
giving
\[
T_{a_j}(a_{Q},g)(x) = \int_{(\mathbb{R}^n)^2} K_j(x,x-y,x-z) a(y) g(z)\, dy dz.
\]

Define $K_j^{(\alpha,0)}(x,y,z) = \partial_y^\alpha K_j(x,y,z)$ and
\begin{align*}
h_j^{(Q,L)}(x) &= 2^{-j\rho n/2} \ell(Q)^{L-n/p} \sum_{|\alpha|=L} \sum_{|\beta| \leq M} \sum_{|\gamma| \leq M'} \int_{\substack{y \in Q \\ 0<t<1}} \\
&\qquad \times \left\|(2^{j\rho}(x-ty))^\beta (2^{j\rho}z)^\gamma K_j^{(\alpha,0)}(x,x-ty,z)\right\|_{L^2_z}\, dy dt.
\end{align*}
The size condition on $a_{Q}$ implies
\[
|T_{a_j}(a_{Q},g)(x)| \lesssim (1+2^{j\rho}|x|)^{-M} h_j^{(Q,L)}(x) \|g\|_{L^\infty}.
\]

Omitting the moment condition on $a_{Q}$, one similarly obtains
\begin{equation}\label{eq:4.9}
|T_{a_j}(a_{Q},g)(x)| \lesssim (1+2^{j\rho}|x|)^{-M} h_j^{(Q,0)}(x) \|g\|_{L^\infty}
\end{equation}
with
\begin{align*}
h_j^{(Q,0)}(x) &= 2^{-j\rho n/2} \ell(Q)^{-n/p} \sum_{|\beta| \leq M} \sum_{|\gamma| \leq M'} \int_{y \in Q} \\
&\qquad \times \left\|(2^{j\rho}(x-y))^\beta (2^{j\rho}z)^\gamma K_j(x,x-y,z)\right\|_{L^2_z}\, dy.
\end{align*}
By H\"older's inequality with $1/p = 1/q + 1/2$,
\[
\left\|T_{a_j}(a_{Q},g)\right\|_{L^p}
\lesssim \left\|(1+2^{j\rho}|\cdot|)^{-M}\right\|_{L^q} \min\left\{\|h_j^{(Q,L)}\|_{L^2}, \|h_j^{(Q,0)}\|_{L^2}\right\}.
\]
Thus the proof reduces to establishing
\begin{equation}\label{eq:4.8}
\|h_j^{(Q,L)}\|_{L^2} \lesssim \left(2^j \ell(Q)\right)^{L-n/p+n} 2^{j\rho n(1/p-1/2)}
\end{equation}
and
\begin{equation}\label{eq:4.10}
\|h_j^{(Q,0)}\|_{L^2} \lesssim \left(2^j \ell(Q)\right)^{-n/p+n} 2^{j\rho n(1/p-1/2)}.
\end{equation}

By Minkowski's inequality for integrals,
\begin{align*}
\|h_j^{(Q,L)}\|_{L^2} &\leq 2^{-j\rho n/2} \ell(Q)^{L-n/p} \sum_{|\alpha|=L} \sum_{|\beta| \leq M} \sum_{|\gamma| \leq M'} \int_{\substack{y \in Q \\ 0<t<1}} \\
&\qquad \times \left\|(2^{j\rho}(x-ty))^\beta (2^{j\rho}z)^\gamma K_j^{(\alpha,0)}(x,x-ty,z)\right\|_{L^2_{x,z}}\, dy dt.
\end{align*}
Then (\ref{E27}) in Section 5 gives (\ref{eq:4.8}), while (\ref{eq:4.10}) follows by the same reasoning.

\end{document}